\documentclass[a4paper,11pt]{amsart}
\usepackage[english]{babel}
\usepackage{amsfonts}
\usepackage{amsthm}
\usepackage{amsmath}
\usepackage{amssymb}
\usepackage{enumerate}
\usepackage{mathtools}
\usepackage{tikz}
\usepackage{tikz-cd}
\usepackage[all]{xy}
\usepackage{graphicx}
\usepackage{adjustbox}
\usepackage{graphicx}
\usepackage{extpfeil}
\usepackage{comment}
\usepackage{float}
\usepackage[labelformat=empty]{caption}
\usepackage[toc]{appendix}
\usepackage[left=2.5cm,top=2.5cm,right=2.5cm,bottom=2.5cm]{geometry}
\usepackage{hyperref}
\usepackage{resizegather}
\usepackage[activate={true, nocompatibility}, final, tracking=true, kerning=true, spacing=true, factor=1100, stretch=10, shrink=10]{microtype}
\usepackage{stackengine}
\usepackage{cleveref}
\usepackage{colonequals}
\usepackage{enumitem}

\theoremstyle{plain}
\newtheorem{thm}{Theorem}[section]
\newtheorem{coroll}[thm]{Corollary}
\newtheorem{lemma}[thm]{Lemma}

\newtheorem{prop}[thm]{Proposition}

\newtheorem{conj}[thm]{Conjecture}

\newtheorem*{thm*}{Theorem}

\theoremstyle{remark}
\newtheorem{remark}[thm]{Remark}
\newtheorem{example}[thm]{Example}
\newtheorem{defn}[thm]{Definition}

\newcommand\hollowslash{\setbox0=\hbox{/}\def\holwd{3pt}%
  \stackengine{-.3pt}{/}{\rlap{\kern-1pt\rule{\holwd}{.4pt}}}{O}{r}{F}{F}{S}%
  \kern\dimexpr\holwd-\wd0-.2pt\relax%
  \stackengine{-.4pt}{/}{\llap{\rule{\holwd}{.4pt}\kern-1pt}}{U}{l}{F}{F}{S}%
}

\tikzset{
  symbol/.style={
    draw=none,
    every to/.append style={
      edge node={node [sloped, allow upside down, auto=false]{$#1$}}}
  }
}

\microtypecontext{spacing=nonfrench}

\DeclareMathOperator{\Sym}{Sym}

\DeclareMathOperator{\Alb}{Alb}
\DeclareMathOperator{\alb}{alb}
\DeclareMathOperator{\Spec}{Spec}

\newcommand{\bQ}{\mathbb{Q}}

\newcommand{\nn}{\mathbb{N}}
\renewcommand{\P}{\mathbb{P}}
\newcommand{\pp}{\mathbb{P}}
\newcommand{\bR}{\mathbb{R}}
\newcommand{\bZ}{\mathbb{Z}}
\newcommand{\Pic}{\operatorname{Pic}}

\newcommand{\rk}{\operatorname{rk}}
\newcommand{\Kum}{\operatorname{Kum}}

\newcommand{\dendegs}{\delta}
\newcommand{\potdendegs}{\wp}
\newcommand{\ind}{\operatorname{ind}}
\newcommand{\gon}{\operatorname{gon}}
\newcommand{\im}{\operatorname{Im}}
\newcommand{\red}{\operatorname{red}}
\newcommand{\AV}{\operatorname{AV}}
\newcommand{\effind}{\operatorname{eff-ind}}

\newcommand{\defi}[1]{\textsf{#1}} 

\usepackage[normalem]{ulem}
\usepackage{pbox}

\usepackage{biblatex}
\title{Density of algebraic points on products of curves}

 \author[J.~Berg]{Jennifer Berg}
 \address{Jennifer Berg: Bucknell University, Department of Mathematics, Lewisburg, PA 17837, USA}
\email{jsb047@bucknell.edu}
 \urladdr{\url{https://sites.google.com/view/jenberg}}

 \author[Y.~Fu]{Yu Fu}
 \address{Yu Fu: Caltech, Department of Mathematics, 1200 E California Blvd, Pasadena, CA 91125, USA}
 \email{yufu@caltech.edu}
 \urladdr{\url{https://sites.google.com/wisc.edu/jerryyfu/home}}

 \author[E.~Gazaki]{Evangelia Gazaki}
 \address{Evangelia Gazaki: University of Virginia, Department of Mathematics, 141 Cabell Drive, Charlottesville, VA 22904, USA}
\email{eg4va@virginia.edu}
 \urladdr{\url{https://sites.google.com/view/valiagazakihomepage/home}}

 \author[M.~Porzio]{Morena Porzio}
 \address{Morena Porzio: Department of Mathematics, Columbia University, New York, NY 10027, USA}
 \email{mp3947@columbia.edu}
 \urladdr{\url{https://www.math.columbia.edu/~morenaporzio/}}

 \author[J.~Rawson]{James Rawson} 
 \address{James Rawson: Mathematics Institute, University of Warwick, Coventry, CV4 7AL, United Kingdom}
 \email{james.rawson@warwick.ac.ak}
 \urladdr{\url{https://warwick.ac.uk/fac/sci/maths/people/staff/rawson/}}

 \author[I.~Vogt]{Isabel Vogt}
\address{Isabel Vogt: Brown University, Department of Mathematics, Box 1917, 151 Thayer Street, Providence, RI 02912, USA}
 \email{ivogt.math@gmail.com}
 \urladdr{\url{https://www.math.brown.edu/ivogt/}}

\begin{document}

\begin{abstract}
    In this paper, we initiate the systematic study of density of algebraic points on surfaces.  We  give an effective asymptotic range in which the density degree set  has regular behavior dictated by the index.  By contrast, in small degree, the question of density is subtle and depends on the arithmetic of the curves.  We give several explicit examples displaying these different behaviors, including products of genus \(2\) curves  with and without dense quadratic points. 
    These results for products of curves have applications to questions about algebraic points on closely related surfaces, such as rank growth on abelian surfaces and bielliptic surfaces.
\end{abstract}

\maketitle

\section{Introduction}

Let \(X\) be a smooth, projective, and geometrically connected scheme defined over a number field \(k\).  
We will refer to such a scheme as a \defi{nice variety}.
Write
\(\dendegs(X/k)\) for the \defi{density degree set} of \(X/k\): the set of natural numbers \(d\) for which the closed points of degree $d$ are Zariski dense in $X$. 
We denote by $\nn$ the set of all positive integers. 
For a nice curve $C$ of genus $g$, the structure of the set \(\dendegs(C/k)\) is known to reflect the geometry of \(C\).  
For instance, $\min \dendegs(C/k)$ is related to the gonality of \(C\) \cite[Corollary 5.2.2]{VV}, and 
\begin{equation}\label{eq:boundgenusdendegsetforcurve}
    \dendegs(C/k) \cap \nn_{ \geq 2g} =  \ind(C/k) \mathbb{N} \cap \nn_{ \geq 2g},
\end{equation} 
where \(\ind(C/k)\) denotes the \defi{index} of \(C\)
\cite[Proposition 5.1.1]{VV}.

Far less is known when the dimension of \(X\) is at least \(2\).  
In this paper, we initiate the study of Zariski density of algebraic points on surfaces in a case that builds upon our understanding for curves. The primary focus is the product of two nice curves \(X = C \times D\) over \(k\).  
The structure of \(\dendegs(C \times D/k)\) is constrained by the structures of \(\dendegs(C/k)\) and \(\dendegs(D/k)\): if the genus of one of \(C\) or \(D\) is at most 9 we show that
\begin{equation}\label{eq:basic_range}
\dendegs(C/k)\cdot  \dendegs(D/k) \subseteq \dendegs(C \times D/k) \subseteq \dendegs(C/k) \cap \dendegs(D/k).
\end{equation}
For curves of arbitrary genera, the upper bound still holds but the lower bound depends on degrees of the \(\mathbb{P}^1\)-parameterized points (see \Cref{inclusionsprod}).

Asymptotically, it is known that for any nice variety \(X\), the density degree set \(\dendegs(X/k)\) contains all sufficiently large multiplies of \(\ind(X/k)\) \cite[Proposition B.0.1]{VV}, though this result is not effective.  Our first main result is an analogue of \eqref{eq:boundgenusdendegsetforcurve}
for \(X = C \times D\).  For simplicity, we state the result in the case that \(C\) and \(D\) both have a \(k\)-point; the general result is Theorem~\ref{thm:generalN}.

\begin{thm}[{Corollary~\ref{coroll:inclusionofN(C,D)}}]\label{thm:asym}
    Let $C$ and $D$ be nice curves over $k$ of genera $g_C$ and $g_D$. Assume further that $C$ and $D$ both have $k$-rational points, then 
    \[
    \nn_{ \geq 2(3g_D g_D + g_C + g_D) } \subseteq \dendegs(C \times D / k).
    \]
\end{thm}

For \(g_C, g_D \leq 2\), we can sharpen this containment using the explicit geometry of low genus curves. As above, we state our result when \(C\) and \(D\) both have a \(k\)-point.

\begin{thm}[{Theorem~\ref{Ec:3,5,7}, Theorem~\ref{thm:nngeq9indeltaCxE}, \Cref{prop:summary1}}]\label{thm:preciseresults}
Suppose that \(C\) and \(D\) are nice curves of genera \(g_C\) and \(g_D\) and assume that $C$ and $D$ both have $k$-rational points.
\begin{enumerate}[label = (\roman*)]
\item\label{main:g1} If \(g_C = g_D = 1\), then
$\nn \setminus \{1, 2\} \subseteq \dendegs(C \times D/k) \subseteq \nn$.
\item \label{main:g1g2} If \(g_C = 1\) and \(g_D = 2\), then
\(
    \nn_{\geq 2} \setminus \{2, 3, 5, 7\} \subseteq \dendegs(C\times D/k) \subseteq \nn_{\geq 2} .
\)
\item \label{main:g2} If \(g_C = g_D = 2\), then
\(\nn_{\geq 2} \setminus \{2, 3, 5, 6, 7, 9, 11\} \subseteq \dendegs(C\times D/k) \subseteq \nn_{\geq 2} .\)
\end{enumerate}
\end{thm}

\begin{remark}
In each case of Theorem~\ref{thm:preciseresults}, the results can be sharpened further with additional assumptions on the curves.

If \(C\) and \(D\) are elliptic curves, we show in \Cref{cor:prod_ell_curve_pos_rank}  that $1\in \dendegs(C \times D/k)$ if and only if both curves have positive rank. In Theorem~\ref{Ec w/o bad j-invariant} we show that \(2 \in \dendegs(C \times D/k)\) provided, for instance, that one of \(C\) or \(D\) has \(j\)-invariant different from \(0, 1728\). 

In Theorem~\ref{thm:preciseresults} part~\ref{main:g1g2}, if $C$ is an isogeny factor of the Jacobian of $D$, we show in Proposition~\ref{isogenyfactor} that $\dendegs(C \times D / k) = \dendegs(D / k)$. 
Further results depend on whether $D$ has a rational Weierstrass point (see \Cref{prop:index1genus2andE} and Theorem~\ref{thm:nngeq9indeltaCxE}).

Finally, when both curves have genus $2$, our general bounds depend on the index of each of the curves and whether the corresponding product surface has index \(1\), $2$ or $4$
(cf.~\Cref{prop:summary1} and \Cref{cor:summary2}).
\end{remark}

To prove Theorem~\ref{thm:asym}, we construct explicit collections of curves \(X_\gamma\) of bounded genus in $C \times D$ passing through fixed sets of closed points, and such that their union is dense in \(C \times D\). 
Consequently, applying the bound \eqref{eq:boundgenusdendegsetforcurve} on each curve, we conclude that for any $d \geq 2 g(X_\gamma)$, the surface has dense degree \(d\) points. To prove Theorem~\ref{thm:preciseresults}, we both construct lower genus families and exploit the fact that \eqref{eq:boundgenusdendegsetforcurve} can be sharpened with more explicit knowledge of the divisors on \(X_\gamma\). 

\subsection{Small Degree Points}
The method of covering the surface with curves and applying the previous strategy, which relies on the Riemann--Roch Theorem, is impractical or impossible for points of very small degree.  Instead, we construct these points by pulling back points from other varieties, which requires proving density of rational points on surfaces.  This is quite difficult and entangled with several deep conjectures in number theory. For example, 
we show in \Cref{thm:parity_3} that the Parity Conjecture sometimes implies that \(3 \in \dendegs(C\times D/k)\) when \(g_C=1\) and \(g_D = 2\).

For quadratic points on \(C \times D\) with \(g_C, g_D \leq 2\), 
we consider the quotient surface \(S \colonequals (C \times D)/\iota\), where \(\iota\) is a hyperelliptic involution.  
Since \(C \times D \to S\) is a finite map of degree \(2\), if \(1 \not\in \dendegs(C \times D/k)\) and \(1 \in \dendegs(S/k)\), then we can conclude via pullback that \(2 \in \dendegs(C \times D/k)\).
(In fact, if $\Pic^0_C$ and $\Pic_D^0$ both have rank $0$, the density of quadratic points on \(C \times D\) can be reduced to the density of rational points on finitely many surfaces of the form \((C \times D)/\iota\).)

In low genus, the surface \(S\) has interesting geometry.
When \(C\) and \(D\) are elliptic curves, \(S\) is the Kummer $K3$ surface associated to the abelian surface $C\times D$.
When $g_C=1, g_D=2$, \(S\) is a quadratic twist elliptic fibration over $\mathbb{P}^1$. 
Conjecturally, all such surfaces have dense rational points \cite[Conjecture 1.1]{DeJard}. 
In Section~\ref{sec:quadraticpoints} we give some examples when this conjecture can be verified.

Lastly, when both \(C\) and \(D\) have genus $2$, the surface \(S\) is of general type, and the Bombieri-Lang conjecture predicts that it does not have dense rational points. 
In \Cref{prop: 2notindelta}, we give an example, based on an idea of Adam Logan, of genus \(2\) curves \(C\) and \(D\) with \(\dendegs(C/k) = \dendegs(D/k) = 2 \nn\) for which \(2 \notin \dendegs(C \times D/k)\) despite the fact that there exist degree \(2\) points on \(C \times D\). 

On the other hand, when \(\Pic^0_C\) or \(\Pic^0_D\) has positive rank, there can be additional sources of quadratic points.  For example, we show in Proposition~\ref{CxC} that the quadratic points are dense on a self-product $C\times C$ of a genus $2$ curve with $\Pic_C^0$ of positive rank by pulling back rational points from $\Pic_C^0$.

Some of the arithmetic complexity is eased by considering the \defi{potential density degree set}.  Recall that for a variety \(X/k\), an integer \(d\) is in the potential density degree set \(\potdendegs(X/k)\) if there exists a finite extension \(k'/k\) for which \(d \in \dendegs(X/k')\).  If \(X = C \times D\) is a product of curves and \(C\) has genus \(0\) or \(1\), we have \(\potdendegs(C \times D/k) = \potdendegs(D/k)\). In the case of products of genus \(2\) curves, we give a nearly complete description of the potential density degree set in Theorem~\ref{thm:pot}.

\subsection{Applications to non-product surfaces}
One of the main reasons for studying product surfaces is their close relation to the more well-studied setting of curves. 
Another reason is their closeness to several other interesting families of surfaces. Our methods and results can be applied to principally polarized abelian surfaces and to biellliptic surfaces in particular.

In both cases, we make use of the fact they are covered by a product of curves.  One subtlety is that the image of a degree \(d\) point under a finite map may be a smaller degree point.  To circumvent this issue, we  utilize constructions of points where we can guarantee that the residue field is a primitive extension of \(k\). 

\begin{thm}[{cf.~Theorem~\ref{thm:Jac of genus 2}, Theorem~\ref{thm:isogeneousabsurf}}]
    If \(A/k\) is an abelian surface that is isogenous (over \(k\)) to a principally polarized abelian surface, then
    \[\nn_{\geq 3} \subseteq \dendegs(A/k).\]
    If, furthermore, \(A\) is isogenous to the Jacobian of a genus \(2\) curve \(C/k\), then \(2 \in \dendegs(A/k)\).
\end{thm}

These density statements for abelian surfaces can be re-interpreted as statements about the behavior of ranks of these abelian varieties over field extensions (see \Cref{cor:simultaneousrankjumps}). This interpretation is one of the key steps for the following description of the density degree set of a bielliptic surface.

\begin{thm}[{\Cref{cor:biellipticsurface}}]
    Let $S = (E_1 \times E_2) / G$ be a bielliptic surface over $k$, where $G$ acts on $E_1$ via translations by $k$-rational points and on $E_2$ such that $E_2/G\simeq \mathbb{P}^1$. Then,
    $$\mathbb{N} \setminus \{1, 2\} \subseteq \dendegs(S/k) \subseteq \dendegs(E_1 / k).$$
\end{thm}

\begin{remark}
One of the main properties of the density degree set of curves is that it is closed under multiplication by positive integers \cite[Section 5.2.1]{VV}.
Whether the same property holds for higher dimensional varieties is still unknown.
However, in many cases that we consider, our sharper description of the density degree set implies that it satisfies this multiplicativity property.
\end{remark}

\subsection{Outline}
The paper is organized as follows. In Section 2, we recall what is known about the density degree sets of curves.  We also prove the upper and lower bounds on the density degree set of a product surface given in \eqref{eq:basic_range}. Section 3 is devoted to describing our general strategy and proving Theorem~\ref{thm:asym}. Products of elliptic curves and bielliptic surfaces are discussed in Section 4. Section 5 concerns the product of an elliptic curve and a curve of genus 2. Products of two genus 2 curves, and also Jacobians of genus 2 curves are studied in Section 6. 
In this section we also study the potential density degree set for products of low genus curves. Section 7 extends the description of the density degree set of a principally polarized abelian surfaces via isogenies.

\subsection{Notation} Unless otherwise stated, throughout this paper $k$ will be a number field. All varieties will be defined over $k$ and will be assumed to be nice (i.e. smooth, projective, geometrically connected). 
The index of a variety $X$ will be denoted by $\ind(X/k)$.

\subsection*{Acknowledgments}

This work began at the American Institute of Mathematics workshop ``Degree \(d\) points on algebraic surfaces'' in March 2024.  We would like to thank  the organizers Nathan Chen and Bianca Viray, as well as all of the staff at AIM and the National Science Foundation for their support.
We would also like to thank Niven Achenjang, Adam Logan, Samir Siksek, and  Rosa Winter for helpful conversations.
E.G.~was partially supported by NSF grant DMS-2302196. 
J.R.~was supported by UK EPSRC grant EP/W523793/1.  I.V.~was supported in
part by NSF grants DMS-2200655 and DMS-2338345.

\section{Background and first computations}\label{bounds} 

\subsection{Density and potential density degree sets}\label{densitybackground} We recall the following definitions \cite[Definitions 1.2.1, 1.2.3]{VV}. 
\begin{defn}
    Let $X$ be a nice variety over a field $k$. The density degree set $\dendegs(X/k)$ is the set of all positive integers $d$ such that the degree $d$ points on $X$ are Zariski dense. 
\end{defn}
The set $\dendegs(X/k)$ is sensitive to the ground field $k$. For example, if $k=\bar{k}$, then $\dendegs(X/k)=\{1\}.$ If $k$ is not algebraically closed, then for extensions $k'/k$, there is no general containment between $\dendegs(X/k')$ and $\dendegs(X/k).$ For instance, if $X/\bR$ is a conic without a rational point, then $\dendegs(X/\mathbb{R})=\{2\}$, while $\dendegs(X/\mathbb{C})=\{1\}$. For this reason, the following definition makes sense.
\begin{defn}\label{def:potentialdensity}
    Let $X$ be a nice variety over a field $k.$ The potential degree set $\wp(X/k)$ is the union of $\dendegs(X/k')$ as $k'/k$ ranges over all possible finite extensions.
\end{defn}
\subsection{Density degrees of curves}
From now on we assume that $k$ is a number field and \(C/k\) is a nice curve. We recall briefly some of the structure of \(\dendegs(C/k)\).  More details can be found in \cite{VV}.  A degree \(d\) point is \(\mathbb{P}^1\)-parameterized if it is the fiber of a degree \(d\) morphism \(C \to \mathbb{P}^1\).  A degree \(d\)-point \(P\) is \(\AV\)-parameterized if there is a positive-rank abelian subvariety \(A \subset \Pic^0_{C/k}\) such that the translate \([P] + A\) is contained in the image of the Abel--Jacobi map \(\Sym^d_C \to \Pic^d_C\).  The curve \(C/k\) has infinitely many points of degree \(d\) if and only if it has a point of degree \(d\) that is either \(\mathbb{P}^1\)- or \(\AV\)-parameterized \cite[Theorem 4.4.3]{VV}.  Consequently,
the density degree set $\dendegs(C/k)$ decomposes
\[\dendegs(C/k)=\dendegs_{\mathbb{P}^1}(C/k)\cup\dendegs_{\AV}(C/k),\]
where $\dendegs_{\AV}(C/k)$ is the set of degrees of \(\AV\)-parameterized points\cite[Corollary 5.0.1, Definition 5.0.2]{VV}.

The main properties of  $\dendegs(C/k)$ that we will need in this article are the following:

\begin{enumerate}
    \item\label{multiplicativity} \(\dendegs(C/k)\) is closed under multiplication by positive integers \cite[Section 5.2.1]{VV}.
    \item\label{2g_bound} \(\dendegs(C/k)\) contains all multiples of the index of \(C\) that are at least \(\max(2g, 1)\) \cite[Proposition 5.1.1]{VV}.
    \item\label{gp1_bound} Every closed point on \(C\) of degree at least \(g+1\) is \(\mathbb{P}^1\)-parameterized \cite[Corollary 4.1.2]{VV}.
    \item\label{finite_extension_P1} If \(k'/k\) is a finite extension, then \(\dendegs_{\mathbb{P}^1}(C/k) \subset \dendegs_{\mathbb{P}^1}(C/k')\) \cite[Lemma 5.5.2]{VV}.
\end{enumerate}

The issue at the heart of property \eqref{finite_extension_P1} which does not allow us to immediately generalize this for the \(\AV\)-parametrized points is that for a degree \(d\) point \(P\), the base change \(P_{k'}\) may become reducible.  The way that \(P_{k'}\) splits as a sum of lower degree points is recorded by a partition of \(d\).  If every partition of \(d\) contains an integer that divides \(d\), then the pigeonhole principle and \eqref{multiplicativity} imply that again \(d \in \dendegs(C/k')\). 
If, instead, every partition that does not contain an integer dividing $d$ has greatest common divisor 1, the same conclusion holds. By \cite[Proposition 2.2.1]{VV}, if one of these partitions arises infinitely often, then over the Galois closure, $k''$, of $k'$, we have $1 \in \dendegs(C / k'')$. This implies that $C$ is genus 0 or 1, and applying \eqref{gp1_bound} and \eqref{finite_extension_P1}, $d \in \dendegs(C / k')$. The first $d$ for which this does not hold is 10.
In particular, combining \eqref{gp1_bound} and \eqref{finite_extension_P1} we see that if \(C\) is a nice curve of genus at most \(9\), then for all finite extensions \(k'/k\), we have \(\dendegs(C/k) \subset \dendegs(C/k')\).

In low genus, we can completely determine the density degree set in terms of the index of the curve and arithmetic of low degree points.

\begin{lemma}\label{lemma:deltagenus2}
Let $C$ be a nice genus $g$ curve over $k$.
\begin{enumerate}
    \item If \(g=1\), then
    \[\dendegs(C/k) = \begin{cases}
        \nn & \textup{if \(C(k)\) is infinite.}\\
        \big(\ind(C/k) \nn \big) \cap \nn_{\geq 2} & \textup{if \(C(k)\) is finite.} 
    \end{cases}\]
    In particular, if \(C\) is an elliptic curve, then \(\nn_{\geq 2} \subseteq \dendegs(C/k)\), with equality if and only if \(\rk C(k) = 0\).
    
\item If \(g=2\), then   
    \[
        \dendegs(C/k) = 
        \begin{cases} 
        2\nn & \textup{if}\  \ind(C/k) = 2, \\
        \nn_{\geq 2} &  \textup{if}\  \ind(C/k) = 1\ \textup{and}\ \exists\, x\in C \ \textup{of degree}\ 3, \\
        \{2\} \sqcup \nn_{\geq 4} & \textup{otherwise}.
        \end{cases}
        \]
        Moreover, if $C$ has index $1$ but has no degree $3$ points (namely the third case), then $C(k)\neq \varnothing$.
        \end{enumerate}
\end{lemma}

\begin{proof} \hfill
\begin{enumerate}
    \item By property \eqref{2g_bound}, we have \(\dendegs(C/k) \cap \nn_{\geq 2} = \big(\ind(C/k)\nn\big) \cap \nn_{\geq 2}\).  The only question, therefore, is whether \(1\) is in \(\dendegs(C/k)\), which is precisely the question of whether \(C(k)\) is finite or infinite.
    \item Faltings' Theorem implies that \(1 \not\in \dendegs(C/k)\).  Since a genus \(2\) curve has gonality \(2\), we have \(2 \in \dendegs_{\mathbb{P}^1}(C/k)\) and \(\ind(C/k)\mid 2\).  By property \eqref{2g_bound}, we have \(\dendegs(C/k) \cap \nn_{\geq 4} = \big(\ind(C/k)\nn\big) \cap \nn_{\geq 4}\).  The only question, therefore, is whether \(3\) is in \(\dendegs(C/k)\), which is only possible if \(\ind(C/k)=1\). 
    If $P\in C$ is a point of degree $3$, then it is \(\mathbb{P}^1\)-parameterized by property \eqref{gp1_bound}, and $3$ is in $\dendegs(C/k)$.

    In the case when $\ind(C/k) =1$ but there is no degree $3$ point, and so \(3 \not\in \dendegs(C/k)\), let $\gamma$ be a $0$-cycle on $C$ of degree $1$. 
    By the Riemann--Roch theorem, $\gamma + K_C$ is equivalent to an effective $0$-cycle $\gamma'$ of degree $3$. Since $C$ has no degree $3$ point, the cycle $\gamma'$ contains a $k$-point. \qedhere
\end{enumerate}
\end{proof}

As in \cite[Example 5.1.3]{VV}, there exist genus \(2\) curves with $\dendegs(C/k)=\{2\}\cup\nn_{\geq 4}$, demonstrating that the bound in \eqref{2g_bound} is sharp. However, the asymptotic bound $\max(2g,1)$ can be sharpened if the curve is not hyperelliptic, as the following Lemma shows.

\begin{lemma}\label{curve_asymptotic}
    Let \(C\) be a non-hyperelliptic curve of genus \(g \geq 3\) with index \(1\).
    \begin{enumerate}
        \item If \(C(k) \neq \emptyset\), then \(2g - 3 \in \dendegs(C/k)\).
        \item If \(\#C(k) \geq 2\) or \(C(k) = \emptyset\), then \(2g - 1\in \dendegs(C/k)\).
    \end{enumerate}
\end{lemma}
\begin{proof}
    By \cite[Lemma 4.1.4]{VV}, it suffices to exhibit a divisor of each of these degrees whose complete linear system is basepoint free of (projective) dimension at least \(1\).  

    For part (1), let \(P \in C(k)\).   Since \(C\) is nonhyperelliptic of genus \(g \geq 3\), the complete linear system of the canonical is very ample of dimension \(g-1 \geq 2\).  Hence the complete linear system of \(K_C(-P)\) is basepoint free of dimension at least \(1\). 

    For part (2), by Riemann--Roch, every line bundle of degree \(2g-1\) has complete linear system of dimension at least \(g\) (which is assumed to be at least \(3\)) and the only line bundles in \(\Pic_C(\bar{k})\) of degree \(2g-1\) that are not basepoint free are of the form \(K_C(P)\) for \(P \in C(\bar{k})\).

    First assume that \(C(k) = \emptyset\).  Let \(D\) be a divisor of degree \(1\) (which exists since \(C\) has index \(1\) by assumption).  Then \(\mathcal{O}_C(D) \not\simeq \mathcal{O}_C(P)\) for any \(P \in C(\bar{k})\), since then \(D\) would be effective and \(C\) would have a rational point. Thus \(K_C(D)\) is basepoint free and hence satisfies the necessary conditions of \cite[Lemma  4.1.4]{VV}.
    
    Now let \(P, Q \in C(k)\) be distinct points. 
    We claim the divisor $K_C + 2P - Q$ is basepoint free. If not, then $K_C + 2P - Q \sim K_C + R$ for some $R \in C(\bar{k})$. In particular, $2P \sim Q + R$, and as $C$ is not hyperelliptic, equality holds on divisors. This contradicts $P$ and $Q$ being distinct.
\end{proof}

Following in a similar vein, if we have further control on the possible morphisms from the curve to $\mathbb{P}^1$, it is possible to strengthen these results further. One common source of such a constraint is the Castelnuovo--Severi inequality.

\begin{lemma}\label{cs_asymptotic}
Let $C_1$ and $C_2$ be curves of genus $g_1$ and $g_2$ respectively, and assume there exists a map $\phi : C_1 \to C_2$ of degree $n$. 
Let $d$ be an integer with $(n - 1)d < g_1 - ng_2$ and $d \leq g_1 - 2$. 
If there exists an effective divisor $Z_1$ on $C_1$ of degree $d$, such that for all points $P \in C_1(\bar{k})$, and all effective divisors $Z_2$ on $C_2$ with $h^0(Z_2) \geq 2$, $Z_1 + P - \phi^* Z_2$ is not effective, then $2g_1 - 2 - d \in \dendegs(C_1 / k)$.
\end{lemma}
\begin{proof}
    Consider the divisor $K_{C_1} - Z_1$ which has degree $2g_1 - 2 - d \geq g_1$. 
    By Riemann-Roch, since $h^0(Z_1) \geq 1$, we have $h^0(K_{C_1} - Z_1) \geq 2$. 
    It remains to show $K_{C_1} - Z_1$ is base-point free. 
    If $h^0(K_{C_1} - Z_1 - P) = h^0(K_{C_1} - Z_1)$ for some $P \in C_1(\bar{k})$, then $h^0(Z_1) < h^0(Z_1 + P)$. 
    In particular, $h^0(Z_1 + P) \geq 2$. 
    This divisor has degree $d + 1$, and so the induced morphism to $\mathbb{P}^1$ has degree at most $d + 1$. 
    By the Castelnuovo--Severi inequality, as $ng_2 + (n - 1)d < g_1$, this morphism must factor through $\phi$ as $\psi \circ \phi$. 
    Let the poles of $\psi$ be $Z_2$: then $Z_1 + P$ contains $\phi^*(Z_2)$, contradicting the assumption that $Z_1 + P - \phi^*(Z_2)$ is not effective.
    It follows that $K_{C_1} - Z_1$ is base-point free.
\end{proof}

\begin{remark}\label{rmk:applicationofLemma2.5}
    Note that if $g_2 > 0$ and $C_1$ has a $k$-rational point $P$, where $P$ is not the unique point in the fiber $\phi$ above $\phi(P)$, then the lemma applies to $dP$.
\end{remark}

\subsection{Galois groups of algebraic points}

\begin{lemma}\label{lem:Sd_points}
    Suppose that \(C/k\) admits a degree \(d\) divisor \(D\) whose complete linear system \(|D|\) is birationally very ample.  Then \(C\) has a dense set of degree \(d\) points \(P \in C\), such that \(|P| = |D|\) and \(k(P)/k\) is a degree \(d\) extension with Galois group of the Galois closure equal to \(S_d\). 
\end{lemma}

\begin{proof}
Consider the morphism \(\varphi \colon C \to \pp^r\) given by the complete linear system \(|D|\).  Since this is birationally very ample, it is an embedding away from finitely many points and \(r \geq 2\) unless \(d=1\) and \(g=0\); since the statement of the lemma in this latter case carries no content, we will assume that \(r \geq 2\).
By the Uniform Position Theorem \cite[Lemma, Page 111]{ACGH}, the Galois group of the incidence correspondence 
\[I \colonequals \{(P, H) \in \varphi(C) \times (\P^
{r})^\vee :  P\in \varphi(C) \cap H\} \to (\P^{r})^\vee\] 
is the full symmetric group \(S_d\).  By Hilbert's irreducibility theorem, the Galois group of the fiber \(I_t\) remains \(S_d\) for a dense set of \(t \in (\P^{r})^\vee(k)\).  Since \(C\) and \(\varphi(C)\) are birational, there exist infinitely many degree \(d\) points \(P \in C\) for which the Galois group of \(k(P)\) is \(S_d\) and \(|P| = |D|\).
\end{proof}

\begin{remark}\label{rem:asymptotic_Sd}
   By the Riemann--Roch Theorem, every divisor of degree at least \(2g\) on a nonhyperelliptic curve is birationally very ample.  In addition, the canonical linear series is very ample of degree \(2g-2\), and the only divisors of degree \(2g-1\) whose complete linear series are not birationally very ample are of the form \(K_C + P\) for \(P \in C\).
   
   If the requirement that the curve be nonhyperelliptic is dropped, then every divisor of degree at least \(2g + 1\) is very ample.  Thus a nice curve of genus \(g\) has \(S_d\)-points for every \(d \in \left(\ind(C/k)  \nn \right) \cap \nn_{\geq 2g+1}\).
\end{remark}

This result allows us to deduce the existence of extensions where the rank of an elliptic curve jumps.

\begin{lemma}\label{lem:rankgoesup}
Let $E/k$ be an elliptic curve defined over a number field. 
Then for every integer \(d \geq 2\) there exists an extension $L/k$ with $[L:k] = d$ with Galois group (of the Galois closure) equal to \(S_d\) such that
\[\textup{rank }E(L) \geq \textup{ rank }E(k)+1.\]
\end{lemma}

\begin{proof}

When \(d=2\), the complete linear system \(|2\mathcal{O}|\) represents \(E\) as a degree \(2\) cover of \(\P^1\).  By Hilbert's irreducibility theorem, there is a dense set of points \(t \in \P^1(k)\) for which the fiber \(E_t\) is a degree \(2\) point, which necessarily has Galois group \(S_2\).  For \(d\geq 3\), the complete linear system \(|d\mathcal{O}|\) gives an embedding \(E \hookrightarrow \P^{d-1}\) of degree \(d\).  Hence, by Lemma~\ref{lem:Sd_points}, for all \(d \geq 3\), there exist infinitely many degree \(d\) points \(P \in E\) for which the Galois group of \(k(P)\) is \(S_d\) and \(|P| = |d\mathcal{O}|\). By the Northcott property of the canonical height, only finitely many of these can be torsion points on \(E\).

For any such nontorsion \(P\), let \(L = k(P)\).
Consider the trace homomorphism \(E(L) \to E(k)\), which takes the sum in the group law on \(E\) of the \(L/k\) conjugates of a point in \(E(L)\).
A geometric point \(Q\) in \(P_{\bar{k}}\) gives a nontorsion \(L\)-point in \(E(L)\).  The \(L/k\)-conjugates of this point are the other geometric points in \(P_{\bar{k}}\).  Hence the image of \(Q\) under the trace map is the sum of the geometric points of \(P_{\bar{k}}\) in the group law on \(E\).  Since the sum in the group law respects rational equivalence, and \(|P| = |d \mathcal{O}|\), \(Q\) is in the kernel of the trace map.  On the other hand, the trace map is multiplication by \(d\) on \(E(k)\), so the intersection of the kernel with \(E(k)\) consists entirely of \(d\)-torsion points.  Since \(P\) is nontorsion, we have \(\langle P \rangle \cap E(k) = \mathcal{O}\).  Hence \(\rk E(L) > \rk E(k)\). 
\end{proof}

\begin{coroll}\label{cor:k(mx)=k(x)}
    Let $E/k$ be an elliptic curve defined over a number field $k$. 
    For each \(d \geq 2\), there exist infinitely many degree \(d\) points $P\in E$ of infinite order such that \(k([m]P) = k(P)\) for all \(m \in \mathbb{Z}\).
\end{coroll}

\begin{proof}
By Lemma~\ref{lem:rankgoesup}, there exists an \(S_d\) extension \(L/k\) and an infinite order point \(P \in E(L)\) such that \(\langle P \rangle \cap E(k) = \mathcal{O}\).  Hence, for every \(m \in \mathbb{Z} \setminus \{0\}\), the field \(k([m]P)\) is a subfield of \(k(P)\) which is not equal to \(k\).  Since \(S_d\)-extensions are primitive, we must have \(k([m]P) = k(P)\).  Since \(\langle P \rangle \subset E(L)\) is infinite, we obtain the desired infinite collection of points by considering multiples of \(P\).
\end{proof}

\subsection{First results on $\dendegs(C\times D/k)$}

\begin{lemma}\label{bound1} Let $X$ be a nice surface over $k$, $C$ a nice curve over $k$ and $\pi: X\to C$ a surjective map between them. Then 
    $\dendegs(X/k)\subseteq \dendegs(C/k)$ and $\wp(X/k)\subseteq \wp(C/k)$.
\end{lemma}

\begin{proof} The containment $\wp(X/k)\subseteq \wp(C/k)$ will follow from the corresponding one for $\dendegs(X/k)$. 
    Let $d\in \dendegs(X/k)$, then the set $\pi (\{ P\in X(\overline{k}): \deg(P)=d\} )$ is dense. 
    For any $P\in X$, a closed point of degree $d$, we have:
    \[
        k\subseteq k(\pi(P))\subseteq k(P).
    \]
    Let $S_e\coloneqq \{ \pi(P): P\in X$ closed point of degree $d$ and $\deg(k(P))=e \}$. 
    Since $e$ divides $d$, there is an $e$ for which $S_e$ is infinite and therefore dense in $C$. 
    In particular $e\in \dendegs(C/k)$.
    Since $\dendegs(C/k)$ is closed under multiplication, $d\in \dendegs(C/k)$ as well.
\end{proof} 

The above lemma gives the following upper bound for $\dendegs(C\times D/k)$.

 \begin{coroll}\label{cor:intersection}
     Let $C, D$ be nice curves. 
     Then $\dendegs(C\times D / k) \subseteq \dendegs(C/k)\cap \dendegs(D/k)$.
 \end{coroll}

 When we know that the density of degree \(d\) points on \(C\) implies the same for \(C_{k'}\) for any finite extension \(k'/k\), we also get lower bounds on the density degree set of \(C \times D\).

\begin{prop}\label{inclusionsprod}
    Let $C, D$ be nice curves.  Then
     \(\dendegs_{\pp^1}(C/k)\cdot \dendegs(D/k)  \subseteq \dendegs(C\times D / k)\). In particular,
    \[
       \left( \dendegs(C/k)\cap \nn_{\geq g_C + 1}\right) \cdot \dendegs(D/k)  \subseteq \dendegs(C\times D / k).
    \]
    If \(g_C \leq 9\), then \(\dendegs(C/k)\cdot \dendegs(D/k)  \subseteq \dendegs(C\times D / k)\).
\end{prop}

\begin{proof} 
    Let $d\in \dendegs(C/k)$ and $e\in \dendegs(D/k)$.  If either \(g_C \leq 9\) or \(d \in \dendegs_{\pp^1}(C/k)\), then it follows that for every finite extension \(k'/k\) we have \(d \in \dendegs(C/k')\) (see property \ref{finite_extension_P1} and the paragraph proceeding it).  In particular, for any degree \(e\) point \(P \in D\), we have \(d \in \dendegs(C_{k(P)}/k(P))\).  The image of a degree \(d\) point of \(C_{k(P)}\) under the basechange map \(C_{k(P)} \to C\) is a point of degree \(de\).  The points obtained in this way are dense on \(C\times D\) by assumption. 
\end{proof}

\subsection{A strategy to prove density}
One technique for guaranteeing density of degree \(d\) points on a nice surface \(S\) is to find a family of curves \(\{X_t\}_{t \in T}\) covering a dense subset of \(S\).  This opens the door to using the density of degree \(d\) points on each of the curves \(X_t\) (where we have more-developed techniques, such as the Abel--Jacobi map and the Riemann--Roch theorem) to guarantee density of degree \(d\) points on \(S\).  This is made precise in the following.

\begin{lemma}\label{lem:covcurvs}
  Let $S$ be a nice surface over $k$.  Suppose that there exists an infinite set \(T\) and nice curves \(X_t\) over \(k\), for each \(t \in T\), such that
    \begin{enumerate}
        \item\label{assp1} for all $t \in T $, there exists a map \(X_t \to S\) which is birational onto its image \(\im(X_t)\), and
        \item\label{assp2} \(\bigcup_{t \in T} \im(X_t)\) is not contained in a proper closed subvariety of \(S\), and
        \item\label{assp3} \(d \in \dendegs(X_t/k)\) for all \(t \in T\). 
    \end{enumerate}
    Then \(d \in \dendegs(S/k)\).
\end{lemma}
\begin{proof}
    By combining assumptions~\eqref{assp1} and \eqref{assp3}, \(\im(X_t) \subset S\) has dense degree \(d\) points.  Ranging over all \(t \in T\), assumption~\eqref{assp2} guarantees that these cover a dense subset of \(S\).  Thus \(d \in \dendegs(S/k)\) as desired.
\end{proof}

The following proposition gives a special case where Lemma~\ref{lem:covcurvs} can immediately be put into practice. 
\begin{prop}\label{isogenyfactor}
    Let $C$ be a curve of genus $g\geq 1$ and $E$ an elliptic curve over $k$.
    Assume that $E$ is an isogeny factor of $\Pic^0_C$.
    Then $\dendegs(C\times E/k)=\dendegs(C/k)$. 
\end{prop}

\begin{proof}
    Let $\phi$ be an isogeny between $\Pic^0_C$ and $E\times A$, where $A$ is an abelian variety of dimension $g-1$. 
    Let \(d \geq 1\) be the index
    of \(C\) and let \(Q \in \Pic^d_C(k)\) be a rational point.  Consider the composite
    \[
        C \xhookrightarrow[]{ \Delta_C } C^{d} \to  \Pic^{d}_C  \xrightarrow[]{(-) \otimes (-Q)}\Pic_C^0  \overset{\phi}{\to}E\times A.
    \]
    Projection to \(E\) defines a map $f:C\to E$.  
    Explicitly, this map sends 
    \[
    P\mapsto (P,\dots,P)\mapsto \mathcal{O}_C(dP) \mapsto \mathcal{O}_C(dP-Q)
    \]
    and then uses the isogeny $\phi$ and projects on $E$. We claim that this map is nonconstant. We can check this by base change to $\overline{k}$. 
    Over $\overline{k}$, up to translation, we can assume $Q=dQ_0$ where $Q_0\in C(\overline{k})$.
    Therefore the map $C_{\overline{k}}\to \Pic^0_{C_{\overline{k}}}$ obtained by base change sends $P\mapsto \mathcal{O}_C(d(P-Q_0))$, and hence it
     can be described as the composition
    \[
    C_{\overline{k}} \xrightarrow{P\mapsto \mathcal{O}_{C_{\overline{k}}(P-Q_0)} } \Pic_{C_{\overline{k}}}^0 \overset{[d]}{\to} \Pic^0_{C_{\overline{k}}}. 
    \] The image of the above map is a curve $D/\overline{k}$ isomorphic to $C_{\overline{k}}$.  If postcomposing with the projection on $E$ were constant, then the universal property of the Albanese variety would imply that the Albanese variety of $D$ is an abelian variety isogenous to $A$, and not $\Pic^0_{C_{\overline{k}}}$. This contradicts the fact that the Albanese variety is a birational invariant.
    
    Then, for any positive integer $n$ consider the multiplication $[n]: E \to E$.
    Consider the graphs $\Gamma_{[n]\circ f}$ in the product $ C\times E$. These are all isomorphic to $C$ and we claim that their union is not contained in a proper Zariski closed subvariety of \(C \times E\).  Indeed, for any geometric point \(P \in C\) such that \(f(P)\) is nontorsion (of which there are infinitely many), the set \(\bigcup_n (P, [n]_* f(P))\) is infinite.  For any given Zariski closed subvariety \(Z \subset C \times E\), there are only finitely many geometric points \(P \in C\) such that \((P \times E)\cap Z\) is infinite (since \(Z\) can contain only finitely many fibers of projection onto \(C\)).  Thus $\dendegs(C/k)\subseteq \dendegs(C\times E/k)$ by Lemma~\ref{lem:covcurvs}.  Since \(\dendegs(C\times E/k) \subset \dendegs(C/k)\)) by Proposition~\ref{inclusionsprod}, we conclude that equality holds.
\end{proof}

Note that the upper bound of Proposition~\ref{inclusionsprod} is achieved in this case. Indeed, for an elliptic curve $E$ over $k$ we have $\dendegs(E/k)=\nn$ or $\nn_{\geq 2}$, and hence $\dendegs(E/k)\cap\dendegs(C/k)=\dendegs(C/k)$.

\section{Asymptotic Results for $\dendegs(C\times D/k)$} \label{sec: asymptoticresults}
 Throughout this section let \(C\) and \(D\) be nice curves over \(k\) of genera \(g_C\) and \(g_D\).  
In this section we give an explicit construction of a family of curves \(X_\gamma\) with \(\ind(X_\gamma/k) \mid \ind(C \times D/k)\) that will satisfy Lemma~\ref{lem:covcurvs} for the surface \(S = C \times D\).  The density of the degree \(d\) points on the curves in this family can be used to prove \Cref{thm:asym}.  The basic construction is as follows.

\begin{lemma}\label{lem:fiber_product_construction}
    Let \(C\) and \(D\) be nice curves and let \(Z \subset C \times D\) be an effective zero-cycle with projections \(Z_C \subset C\) and \(Z_D \subset D\).
    Let \(f_C \colon C \to \P^1\) and \(f_D \colon D \to \P^1\) be nonconstant morphisms such that \(f_C(Z_C^{\red}) = f_D(Z_D^{\red}) = \infty \in \P^1\).  For any automorphism \(\gamma \colon \P^1 \to \P^1\) fixing \(\infty\), the fiber product
\begin{center}
    \begin{tikzcd}
        X_\gamma \colonequals C \times_{\P^1} D \arrow[d, "\pi_C"] \arrow[rr, "\pi_D"] && D \arrow[d, "f_D"] \\
        C \arrow[r, "f_C"] & \P^1 \arrow[r, "\gamma"] & \P^1
    \end{tikzcd}
\end{center}
is a curve of arithmetic genus \((\deg(f_C) - 1)( \deg(f_D) -1) + \deg(f_C)g_D + \deg(f_D)g_C \).  The inclusion \((\pi_C, \pi_D) \colon X_\gamma \to C \times D\) sends \(X_\gamma\) to a curve on \(C \times D\) through the zero-cycle \(Z\).  Varying \(\gamma\) in any infinite  family, this family of curves sweeps out a dense locus in \(C \times D\).
\end{lemma} 
\begin{proof}
    We compute the genus of \(X_\gamma\) by the Riemann--Hurwitz formula.  The map \(\pi_C \colon X_\gamma \to C\) has degree \(f_D\) and is ramified at the preimages of the ramification points of \(f_D\).  In total we see
    \[2g(X_\gamma) - 2 = \deg(f_D)(2g_C - 2) + \deg(f_C)(2g_D - 2 + 2\deg(f_D)),\]
    and hence
    \[g(X_\gamma) = \deg(f_C) \deg(f_D)- \deg(f_C) - \deg(f_D) + 1 + \deg(f_C)g_D + \deg(f_D)g_C . \]
    For every geometric point \(P \in Z_{\bar{k}}\), we have \(f_C(P) = f_D(P) = \infty \in \P^1\).  Since \(\gamma(\infty) = \infty\) by assumption, this guarantees that \(\gamma(f_C(P)) = f_D(P)\), and hence \(P\) is a point of \(X_\gamma\).
    Hence, the zero-cycle \(Z \subset C \times D\) belongs to \(X_\gamma\) for every choice of \(\gamma\).  
    
    Finally, we must verify that as we vary \(\gamma\), the curves \(X_\gamma \subset C \times D\) move.  Composing with the finite map \((f_C, f_D) \colon C \times D \to \P^1 \times \P^1\) the curve \(X_\gamma\) maps onto the graph of \(\gamma^{-1}\).  Since, varying \(\gamma\), these sweep out \(\P^1 \times \P^1\), we get the analogous result for \(X_\gamma\).
\end{proof}

Using Lemma~\ref{lem:covcurvs}, the basic construction in Lemma~\ref{lem:fiber_product_construction} will let us deduce density results for \(C \times D\) using density results for the curves \(X_\gamma\) covering \(C \times D\).
We would like to apply the bounds on the density degree set of a smooth curve coming from the Riemann--Roch theorem.  
One small subtlety is that the curve \(X_\gamma\) may be singular if the branch divisors of \(\gamma \circ f_C\) and \(f_D\) in \(\P^1\) intersect,
so we actually need to apply the Riemann--Roch theorem on the normalization \(X_\gamma^\nu\).  The following will be used to guarantee that \(X_\gamma^\nu\) has the right index.

\begin{lemma}\label{lem:split_over_node}
Let \(K/k\) be a finite extension.
Suppose that \(P \in C(K)\) is a ramification point of \(f_C\) and \(Q \in D\) is a ramification point of \(f_D\) and \(f_C(P) = f_D(Q) = \infty\).  Assume that either 
\begin{itemize}
    \item The ramification indices of \(f_C\) at \(P\) and \(f_D\) at \(Q\) are relatively prime; or
    \item The points \(P\) and \(Q\) are \(k\)-rational (i.e., \(K = k\)).
\end{itemize}
Then there exist infinitely many distinct automorphisms \(\gamma \colon \P_k^1 \to \P_k^1\) fixing \(\infty\) such that the normalization \(X_\gamma^\nu\) of \(X_\gamma\) has a \(K\)-rational point over \((P,Q)\). 
\end{lemma}
\begin{proof}
    We work in formal local coordinates.  Let \(z_C\) be a formal local coordinate on \(C\) near \(P\) and let \(z_D\) be a formal local coordinate on \(D\) near \(Q\).  Let \(w\) be a formal local coordinate at \(\infty \in \P^1\).  Then formally locally near \(P\) the map \(f_C\) is given by \(w = a_nz_C^n + a_{n+1} z_C^{n+1} + \cdots \in K[[z_C]]\) and near \(Q\) the map \(f_D\) is given by \(w = b_mz_D^m + b_{m+1}z_D^{m+1} + \cdots \in K[[z_C]]\).  Near the point \((P, Q)\), the curve \(X_K \colonequals (C \times_{\P^1} D)_K\) has equation 
    \begin{equation}\label{equationXinformallocalcoordinates}
        \cdots + a_{n+1}z_C^{n+1} + a_nz_C^n = b_mz_D^m + b_{m+1}z_D^{m+1} + \cdots. 
    \end{equation} 
    Let \(s\) be the greatest common divisor of \(n\) and \(m\).  Then we claim that the normalization \(X^\nu\) has a \(K\)-rational point above \((P,Q) \in X(K)\) if \(a_n/b_m\) is a perfect \(s\)th power.  Indeed, in this case, if we let \(n = sn'\) and \(m= sm'\) and \(c^s = a_n/b_m\), then the equation for \(X\) becomes
    \[\cdots + (a_{n+1}/b_m)z_C^{sn'+1} + c^sz_C^{sn'} = z_D^{sm'} + (b_{m+1}/b_m)z_D^{sm'+1} + \cdots\]
    Since \(c^s\) is a perfect \(s\)th power, the power series \(c^s + (a_{n+1}/b_m)z_C + \cdots \) has a formal \(s\)th root \(r_C = r_C(z_C)\in K[[z_C]]^\times\).  Similarly the power series \(1 + (b_{m+1}/b_m)z_D + \cdots\) has an \(s\)th root \(r_D = r_D(z_D)\in K[[z_D]]^\times\).  The equation for \(X_K\) formally locally near \((P,Q)\) is thus equivalently
    \[(r_C z_C^{n'})^s - (r_D z_D^{m'})^s = 0.\]
    This always has a nontrivial factor \(r_Cz_C^{n'} - r_D z_D^{m'} = 0\).  Since \(\gcd(n', m') = 1\), there exist integers \(a\) and \(b\) such that \(an' +bm' = 1\).  We claim that the normalization of \(\Spec K[[z_C, z_D]]/(r_Cz_C^{n'} - r_Dz_D^{m'})\) is isomorphic to \(\Spec K[[t]]\) via the map 
    \begin{equation}\label{eq:normalization_map}
        z_C\mapsto \left(\frac{r_D}{r_C}\right)^at^{m'}, \qquad z_D \mapsto \left(\frac{r_C}{r_D}\right)^bt^{n'}.
    \end{equation}
    Indeed, this is well-defined since, using \(an' +bm' = 1\), we see
    \begin{align*}
        r_C \left(\frac{r_D}{r_C}\right)^{an'}t^{n'm'} &= r_C^{bm'}r_D^{an'}t^{n'm'} \\
        &= r_D\left(\frac{r_C}{r_D}\right)^{bm'}t^{n'm'}.
    \end{align*}
    Combining \(an' +bm' = 1\) and \(r_Cz_C^{n'} = r_D z_D^{m'}\), we have 
    \[r_C^az_C = (r_Cz_C^{n'})^az_C^{bm'} = (r_Dz_D^{m'})^az_C^{bm'} = r_D^a(z_D^az_C^b)^{m'}\]
    The formula in \eqref{eq:normalization_map} implies that \(t = z_D^az_C^b\) and so 
    \(t^{m'} = (r_C/r_D)^a z_C\).  Thus \(t\) is a root of a monic polynomial with coefficients in \(K[[z_C, z_D]]/(r_Cz_C^{n'} - r_Dz_D^{m'})\).  Hence the normalization is equal to \(\Spec K[[t]]\) by the universal property of normalization.  This clearly has a rational point over the singular point \((P,Q) \in X(K)\). 

    If \(s= 1\) (i.e., the ramification indices \(n\) and \(m\) are relatively prime), then \(a_n/b_m\) is always a perfect \(s\)th power, so there is no condition on \(\gamma\) for \(X_\gamma^\nu\) to have a \(K\)-point above \((P,Q)\).  Assume, therefore, that \(K = k\). 
    An automorphism of \(\P^1\) fixing \(\infty\) sends \(w\) to \(w/(\alpha + \beta w)\) for \(\alpha \neq 0\).  Taking \(\alpha \in (b_m/a_n){k^\times}^s\) and \(\beta=0\) yields an infinite family of curves \(X_\gamma\) for which the local equation near \((P,Q)\) is 
    \[\cdots + \alpha (a_{n+1}/b_m)z_C^{sn'+1} + \alpha (a_n/b_m)z_C^{sn'} = z_D^{sm'} + (b_{m+1}/b_m)z_D^{sm'+1} + \cdots\] 
    and so there exists \(c^s \in k^{\times s}\) so that \(\alpha (a_n/b_m) = c^s\).  Hence \(X_\gamma^\nu\) has a \(k\)-rational point above \((P,Q) \in X_\gamma(k)\).
\end{proof}

\subsection{Pointed curves}  The simplest case is when both \(C\) and \(D\) have a \(k\)-rational point.

\begin{coroll}\label{coroll:inclusionofN(C,D)}
     Suppose that \(g_C\geq 1\) and \(g_D \geq 1\) and that \(C(k) \neq \emptyset\) and \(D(k) \neq \emptyset\).  Let \(N(C, D) = 2 \big((\gon(C)-1)(\gon(D)-1) + \gon(C)g_D + \gon(D)g_C\big)\). Then 
     \[ \nn_{\geq N(C, D)} \subseteq \dendegs(C \times D/k).\]
     In particular, since the gonality of a pointed curve of genus \(g\) is at most \(g+1\), we have
     \[ \nn_{\geq 6g_Cg_D + 2g_C + 2g_D} \subseteq \dendegs(C \times D/k).\]
\end{coroll}
\begin{proof}
    Apply Lemma~\ref{lem:fiber_product_construction} with \(Z\) a rational point of \(C \times D\), and \(f_C\) and \(f_D\) the maps realizing the gonalities of \(C\) and \(D\) (with choice of coordinates on \(\P^1\) so that \(Z\) has the same image under each map).  We conclude that the surface \(C \times D\) is covered by curves \(X_\gamma\) of geometric genus at most \((\gon(C)-1)(\gon(D)-1) + \gon(C)g_D + \gon(D)g_C\) that pass through \(Z\). 
 Further, using Lemma~\ref{lem:split_over_node}, we may assume that the normalization \(X_\gamma^\nu\) of \(X_\gamma\) has \(X_\gamma^\nu(k) \neq \emptyset\).  By \cite[Proposition 5.1.1]{VV}, we have that \(\dendegs(X_\gamma^\nu/k)\) contains all integers at least twice the genus of \(X_\gamma^\nu\).  Hence 
 \(\nn_{\geq N(C, D)} \subset \dendegs(X_\gamma^\nu/k)\) for all \(\gamma\).  Since the \(X_\gamma\) cover \(C \times D\) this implies that \(\nn_{\geq N(C, D)} \subset \dendegs(C \times D/k)\).

 The bound on the gonality of a smooth curve \(X\) with \(P \in X(k)\) comes from the Riemann--Roch theorem, since \(h^0(X, (g+1)P) \geq 2\); this linear system may not be basepoint-free, but if there are basepoints, then \(X\) has gonality strictly less than \(g+1\).
\end{proof}

Something similar can be done when just one curve has a point, and the other has index 1.
\begin{coroll}\label{cor:pointedcurvexind1}
    Let $C$ be an index one curve, and $D$ a pointed curve. Suppose that $C$ and $D$ admit morphisms to $\mathbb{P}^1$ of degrees $d_C$ and $d_D$ respectively, with $d_C$ coprime to $2 d_D (g_C - 1)$ (this is possible since $C$ has index 1). Set $N(C, D) = 2\left((d_C - 1)(d_D - 1) + d_C g_D + d_D g_C\right)$, then
    $$\mathbb{N}_{\geq N(C, D)} \subset \dendegs(C \times D / k).$$
\end{coroll}
\begin{proof}
    Let $f_C$ and $f_D$ be maps of degree $d_C$ and $d_D$ respectively. Assume that $\infty$ is not a branch point for $f_C$, and that $f_D$ maps the rational point of $D$ to $\infty$. Under these assumptions, the fiber product $X_{\lambda} := C \times_{f_C, \lambda f_D} D$, has a point of degree $d_C$ (the pre-image of $\infty$). Moreover, for a general $\lambda$, this product is smooth.

    Applying Lemma~\ref{lem:fiber_product_construction}, gives that $X_{\lambda}$ generically has genus $g_{X_\lambda} = (d_C - 1)(d_D - 1) + d_C g_D + d_D g_C = d_C(d_D + g_D - 1) + d_D g_C - d_D + 1$. The index of $X_{\lambda}$ divides $2g_{X_{\lambda}} - 2$ and also divides $d_C$ (as it has a point of degree $d_C$).  Thus it divides their the greatest common divisor, which is by assumption, 1. We conclude by observing that $X_{\lambda}$ must have dense degree $d$ points for all $d \geq 2g_{X_{\lambda}}$.
\end{proof}

\subsection{The general case}

In the general case, we similarly want to find a family of curves of bounded genus on \(C \times D\) each of which has index dividing the index of \(C \times D\). 
We will do this by arranging for these curves to interpolate a zero-cycle on \(C \times D\) composed of points whose greatest common divisor of degrees is the index. We write \(|Z|\) for the support of the zero-cycle \(Z\)  (in particular, every coefficient in \(|Z|\) is \(1\)).

\begin{defn}\label{def:effind}
    The \defi{effective index} \(\effind(C \times D/k)\) of the surface \(C \times D\) is
    \[\min\big(\deg|Z| : Z \in Z_0(C \times D) \text{ and } \deg Z = \ind(C \times D/k) \big).\]
\end{defn}

In Section~\ref{subsubsec:effind}, we give bounds on \(\effind(C \times D/k)\) depending only on \(g_C\) and \(g_D\) in the case that \(\ind(C \times D/k) = 1\). 

\begin{prop}\label{thm:generalN}
    Let \(e = \effind(C \times D/k)\) and let
    \[N(C, D, e) = 2(\max(2g_C, 2g_C-2+e) - 1)( \max(2g_D, 2g_D-2+e) -1) \] 
    \[\qquad \qquad \qquad + 2\max(2g_C, 2g_C-2+e)g_D + 2\max(2g_D, 2g_D-2+e)g_C.\]
    Then \(\nn_{\geq N(C,D, e)} \cap \ind(C \times D/k) \nn \in \dendegs(C \times D/k)\).
\end{prop}
\begin{proof}
    Let \(Z'\) be a zero-cycle on \(C \times D\) with \(\deg(Z') = \ind(C \times D/k)\) and \(\deg |Z'| = \effind(C \times D/k) = e\).  
    We will apply Lemma~\ref{lem:fiber_product_construction} with \(Z = |Z'|\).  Let \(Z_C\) and \(Z_D\) denote the images of \(Z\) under the projections to \(C\) and \(D\) respectively.
    The case that \(C \times D\) has a rational point was already dealt with in Corollary~\ref{coroll:inclusionofN(C,D)} (with a better bound), so we assume that one of \(\deg |Z_C|\) or \(\deg |Z_D|\) is at least \(2\).
    
    Since a hyperplane in projective space can avoid any finite collection of points, there exists an effective divisor \(F_C\) of degree \(2g_C - 2\) such that \(|F_C| \cap |Z_C| = \emptyset\).  Similarly let \(F_D\) be an effective divisor on \(D\) of degree \(2g_D - 2\) such that \(|F_D|\cap |Z_D| = \emptyset\).  Define
    \[Y_C \colonequals \begin{cases} F_C + |Z_C| &: \ \deg|Z_C| \geq 2 \\ F_C + 2|Z_C| & : \ \deg |Z_C| = 1.\end{cases}\]
    Similarly define
        \[Y_D \colonequals \begin{cases} F_D + |Z_D| &: \ \deg|Z_D| \geq 2 \\ F_D + 2|Z_D| & : \ \deg |Z_D| = 1.\end{cases}\]
    By construction, these divisors satisfy 
    \[2g_C \leq \deg L_C \leq \max(2g_C, 2g_C -2 + e), \qquad 2g_D \leq \deg L_D \leq \max(2g_D, 2g_D -2 + e).\]
    By the lower bound on \(\deg Y_C\), we can choose a basepoint free pencil in \(|Y_C|\) giving rise to a map \(f_C \colon C \to \P^1\) whose fiber over \(\infty\) is \(Y_C\); similarly for \(f_D \colon D \to \P^1\) with fiber over infinity \(Y_D\).

    We will apply Lemma~\ref{lem:fiber_product_construction} with these maps \(f_C\) and \(f_D\), which guarantees that the surface \(C \times D\) is covered by curves \(X_\gamma\) of arithmetic genus
    \[(\max(2g_C, 2g_C-2+e) - 1)( \max(2g_D, 2g_D-2+e) -1) + \max(2g_C, 2g_C-2+e)g_D \]
    \[\qquad \qquad+ \max(2g_D, 2g_D-2+e)g_C.\]
    Since \(\deg |Z_C|\) and \(\deg |Z_D|\) cannot both be \(1\), the ramification indices of the maps \(f_C\) and \(f_C\) at the points of \(|Z_C|\) and \(|Z_D|\) are relatively prime.
   Thus, by Lemma~\ref{lem:split_over_node}, the curves \(X_\gamma^\nu\) produced by Lemma~\ref{lem:fiber_product_construction} have points over the same fields as the points of \(Z\).  In particular, since the greatest common divisor of the degrees of the points of \(Z\) is \(\ind(C \times D/k)\) by assumption, the index of each \(X_\gamma^\nu\) divides the index of \(C \times D\).  
    By \cite[Proposition 5.1.1]{VV}, we have that \(\dendegs(X_\gamma^\nu/k)\) contains all integer multiples of \(\ind(X_\gamma^\nu/k)\) at least twice the genus of \(X_\gamma^\nu\).  Since the arithmetic genus of \(X_\gamma\) is an upper bound on the geometric genus of \(X_\gamma^\nu\), and \(\ind(X_\gamma^\nu/k) \mid \ind(C \times D/k)\), we conclude that \(\dendegs(C \times D/k)\) contains all integer multiples of \(\ind(C \times D/k)\) of size claimed. 
\end{proof}

\begin{remark}
    Determining the index and effective index in practice can be challenging. We give two examples of curves which individually have index 2, and their product has no quadratic points, but one pair gives a surface of index 2, and the other a surface of index 4. The absence of quadratic points in both cases can be explained by a local obstruction at 3.
    
    The pair of genus 2 curves, $y^2 = 3(x^6 - x^2 + 1)$, $y^2 = -(x^6 - x^2 + 1)$, give a product with index 2 and effective index is 10. These indices arise as the first curve only has points over extensions of $\mathbb{Q}_3$ with inertia degree divisible by 6 or even ramification degree; the second only has points over fields of even inertia degree. In particular, there can be no quadratic field over which both curves have points, and so the effective index must be at least 10, as the lowest degree effective zero-cycles are degree 4, and for parity reasons, a zero-cycle of degree at least 6 must also appear in the support. The effective index being exactly 10, and the index of the product being 2 comes from the existence of a degree 6 point over a common field: when $y = 0$ for both curves. 
    
    The pair of genus 3 curves, $y^2 = 3(x^8 + x^4 + 2)$, $y^2 = 2(x^4 + x^2 + 2)^2 + 3$ both have index 2, their product has index 4 and the effective index is also 4. The obstructions in this case are similar: a point on the first curve must be defined over a field of ramification index a multiple of 2, or with inertia degree a multiple of 8; a point on the second curve must be defined over a field with inertia degree a multiple of 2. Combining these two local conditions forces that any field over which both curves have a point must either have inertia degree a multiple of 8, or an even inertia degree and even ramification index, and so the field must have degree a multiple of 4. 
\end{remark}

\subsubsection{Bounds on the effective index}\label{subsubsec:effind}
We focus here on the case that either \(C\) or \(D\) has index \(1\); in particular, this covers the case that \(C \times D\) has index \(1\).

\begin{lemma}\label{lem:bound_eff_index_1}
    Suppose that \(\ind(C/k) = 1\).  Then there exists a zero cycle \(Z \subset C \times D\) with \(\deg(Z) = \ind(C \times D/k)\) and 
    \[\deg |Z|\leq 4g_Cg_D + (4g_C + 2) (\ind(D/k)-1) + 2g_C \ind(D/k) + 2g_D  + \ind(D/k).\]
    In particular, if \(\ind(C \times D/k)=1\), then \(\ind(D/k) = 1\) and \(\deg|Z|\leq 4g_Cg_D + 2g_C + 2g_D + 1\).
\end{lemma}
\begin{proof}
    By considering the image of a zero-cycle on \(C \times D\) under the projection to \(D\),
    the index of \(D\) divides the index of \(C \times D\).  Hence, by the Riemann-Roch theorem and the Hilbert irreducibility theorem (as in \cite[Proposition 5.1.1]{VV}), there exist effective divisors \(Z_C, W_C\) on \(C\) of degrees \(g_C\) and \(g_C  + 1\), and effective divisors \(Z_D\) and \(W_D\) on \(D\) of degrees \(n\ind(D/k)\) and \((n+1)\ind(D/k)\), where \(n \in \nn\) is the least integer such that \(n \ind (D/k) \geq g_D\).  In particular, we have the bound \(n \ind (D) \leq g_D -1 + \ind(D/k)\). 

    The pairwise products \(Z_C \times Z_D, Z_C \times W_D, W_C \times Z_D, W_C \times W_D\) are zero cycles on \(C \times D\) of degrees 
    \begin{align*}
        \deg (Z_C \times Z_D) &= n g_C \ind(D/k), \\
        \deg (Z_C \times W_D) &= n g_C \ind(D/k) + g_C \ind(D/k), \\
        \deg (W_C \times Z_D) &= n g_C \ind(D/k) + n \ind(D/k), \\
        \deg (W_C \times W_D) &= n g_C \ind(D/k) + n \ind(D/k) + g_C \ind(D/k) + \ind(D/k).
    \end{align*} 
    The linear combination
    \[Z = Z_C \times Z_D - Z_C \times W_D - W_C \times Z_D + W_C \times W_D\] 
    therefore has \(\deg(Z) = \ind(D/k)\).  Since \(\ind(D/k)\) divides \(\ind(C \times D/k)\) they must be equal, and so \(\deg(Z) = \ind(C \times D/k)\).
    We have 
    \begin{align*}
        \deg |Z| &= 4n g_C \ind(D/k) + 2g_C \ind(D/k) + 2n \ind(D/k) + \ind(D/k) \\
        & \leq 4g_C(g_D -1 + \ind(D/k)) + 2g_C \ind(D/k) + 2(g_D -1 + \ind(D/k)) + \ind(D/k) \\
        &= 4g_Cg_D + (4g_C + 2) (\ind(D/k)-1) + 2g_C \ind(D/k) + 2g_D  + \ind(D/k). \qedhere
    \end{align*}
\end{proof}

\section{Products of elliptic curves}\label{sec:prodellcurves}
In this section we focus on a product $X=E_1\times E_2$ of two elliptic curves over $k$. When $E_1, E_2$ are isogenous over $k$, it follows immediately by Proposition~\ref{isogenyfactor} that $\dendegs(E_1\times E_2/k)=\dendegs(E_1/k)=\dendegs(E_2/k)$. 

Using the group structure of elliptic curves, the density degree set of an arbitrary product of elliptic curves admits the following more straightforward description.

\begin{lemma}\label{lem:deltaE}
    Let \(E_1, \dots, E_n\) be elliptic curves over \(k\).  Then the following are equivalent:
    \begin{enumerate}
        \item\label{part:d_in_den} \(d \in \dendegs(E_1 \times \cdots \times E_n/k)\). 
        \item\label{part:infinite_order} There exists a degree \(d\) point \(P \in E_1 \times \cdots \times E_n\) such that the image of \(P\) under each of the projections \(\pi_j \colon E _1 \times \cdots \times E_n \to E_j\) has infinite order in \(E_j(k(P))\).
    \end{enumerate}
\end{lemma}
\begin{proof}
We first prove \eqref{part:d_in_den} \(\Rightarrow\) \eqref{part:infinite_order}. By the Northcott property of the canonical height, there are finitely many torsion points on \(E_j\) of degree at most \(d\).  Call the set of such points \(T_{d, j}\).  The union \(\bigcup_{j=1}^n\pi_j^{-1} T_{d, j}\) is a Zariski closed subset of \(E_1 \times \cdots \times E_n\).
Thus if \(d \in \dendegs(E_1 \times \cdots \times E_n/k)\), there exist a degree \(d\) point outside this subset, i.e., all of whose projections are nontorsion.

To prove  \eqref{part:infinite_order} \(\Rightarrow\)  \eqref{part:d_in_den}, let \(P_j \colonequals \pi_j(P)\) be the \(j\)th projection and consider the subgroup \(B \colonequals \langle P_1\rangle \times \cdots \times \langle P_n \rangle\).  The union of the points of \(B\) are dense in \(E_1 \times \cdots \times E_n\).  Furthermore, every point in \(B\) is defined over \(k(P)\) since each \(P_j\) is defined over \(k(P)\).  Thus \(P \in B \subset (E_1 \times \cdots \times E_n)(k(P))\).  For each of the (finitely many) proper subextensions \(k \subseteq k_i \subsetneq k(P)\), let \(H_i \colonequals B \cap (E_1 \times \cdots \times E_n)(k_i)\).  Then each \(H_i\) is a subgroup of \(B\) and \(P \not\in \bigcup_i H_i\).  Thus by \cite[Lemma 4.3.8]{VV}, there exists a finite index subgroup \(H \subset B\) such that \((P + H)\cap H_i  = \emptyset\) for all \(i\).  Since \(H\) is finite-index in \(B\), the union of the points of \(P + H\) is also dense in \(X\).  We have arranged that for each \(Q \in P + H\), the residue field is exactly \(k(P)\), which is a degree \(d\) extension of \(k\).  Thus \(d \in \dendegs(E_1 \times \cdots \times E_n/k)\).
\end{proof}

\begin{coroll}\label{cor:prod_ell_curve_pos_rank}
    Let \(E_1\) and \(E_2\) be elliptic curves over \(k\).  If \(\rk E_2(k) >0\), then 
    \[\dendegs(E_1 \times E_2/k) = \dendegs(E_1/k).\]
    In particular, $1\in \dendegs(E_1\times E_2/k)$ if and only if both $\rk E_1(k)$ and $\rk E_2(k)$ are positive.
\end{coroll}
\begin{proof}
    We know that \(\dendegs(E_1 \times E_2/k) \subset \dendegs(E_1/k)\) by Corollary~\ref{cor:intersection}.  We will show the reverse containment.
    Let \(P_2 \in E_2(k)\) be a point of infinite order.  
    By Lemma~\ref{lem:deltaE} with \(n=1\), if \(d \in \dendegs(E_1/k)\), then there exists a degree \(2\) point \(P_1 \in E_1\) having infinite order in \(E_1(k(P_1))\).
    Then the point \((P_1,P_2) \in (E_1 \times E_2/k)\) has residue field \(k(P_1)\) and satisfies Lemma~\ref{lem:deltaE}\eqref{part:infinite_order}.  Thus \(d = [k(P_1):k]\) is in \(\dendegs(E_1 \times E_2/k)\). 
\end{proof}

Given Corollary~\ref{cor:prod_ell_curve_pos_rank}, the main case left to study is when \(\rk E_1(k) = \rk E_2(k) = 0\), in which case we know that \(1 \not\in \dendegs(E_1 \times E_2/k)\).  The case of \(d=2\) is more subtle than \(d \geq 3\), and so we deal with this separately.

\subsection{Quadratic points on a product of elliptic curves}

 The following result is essentially due to Kuwata and Wang \cite{KW93}, restated within our framework. 

\begin{thm}\label{Ec w/o bad j-invariant}
    For any elliptic curves $E_1, E_2$, with $j$-invariants not both 0 or both 1728,
    $2 \in \dendegs(E_1 \times E_2 / k)$. If $E_i[2]\subset E_i(k)$, the assumption on the $j$-invariants can be dropped. 
\end{thm}
\begin{proof} Assume first that the $j$-invariants are not both 0 or 1728. Then we can write $E_1$ as $y_1^2 = x_1^3 + a_1 x_1 + b_1$ and $E_2$ as $y_2^2 = x_2^3 + a_2 x_2 + b_2$, with one of $a_1$ and $a_2$ non-zero, and similarly for $b_1$ and $b_2$. 

Let $X$ denote the normalization of the fiber product of $E_1$ and $E_2$ over \(\P^1\) via  
\[(x_i, y_i) \mapsto \frac{y_i^2}{x_i^3} = 1 + a_i \left(\frac{1}{x_i}\right)^2 + b_i \left(\frac{1}{x_i}\right)^3.\]
Each of these maps factors \(E_i \xrightarrow{(x_i, y_i) \mapsto 1/x_i} \pp^1 \xrightarrow{z \mapsto 1 + b_iz^2 + a_iz^3} \pp^1\).  Therefore \(X\) is a biquadratic cover of the curve \(Z\), the normalization of the fiber product of two $\mathbb{P}^1$s via the intermediate cubic maps. The curve \(Z\) has a birational equation $1 + a_1 z^2 + b_1 z^3 = 1 + a_2 w^2 + b_2 w^3$, which is an irreducible singular plane cubic, and hence isomorphic to \(\pp^1\), since the ratios $a_1 : a_2$ and $b_1 : b_2$ can be assumed to be distinct. This can be ensured by scaling $x_2$ by $\lambda$, which scales $a_1 : a_2$ by $\lambda^4$ and $b_1 : b_2$ by $\lambda^6$.
Let \(Z_i \colonequals Z \times_{\pp^1} E_i\); these are two of the quadratic subfields of the biquadratic map \(X \to Z\).  Both maps \(Z_i \to Z\) are branched over pairs \((z, w)\) with \(z\) a root of \(b_1z^3 + a_1z^2 + 1\) and \(w\) a root of \(b_2w^3 + a_2w^2 + 1\).  The map \(Z_1 \to Z\) is branched over the one additional point \((0, -a_2/b_2)\) and the map \(Z_2 \to Z\) is branched over the one additional point \((-a_1/b_1, 0)\).  Let \(Z_3 \to Z\) denote the third quadratic subfield of \(X \to Z \simeq \pp^1\); this map is branched over \((0, -a_2/b_2)\) and\((-a_1/b_1, 0)\).  Consequently \(Z_3\) has genus \(0\) and has a rational point since the branch points on \(Z\) are rational.  Thus \(Z_3 \simeq \pp^1\), and we learn that \(X\) is hyperelliptic and so has dense degree \(2\) points.
Viewing $E_1 \times E_2$ as an abelian surface, the union of the pushforwards $[n]_* X$ is a dense set, and so $E_1 \times E_2$ has dense quadratic points by Lemma~\ref{lem:covcurvs}.

    Next suppose $E_i[2]\subset E_i(k)$ for $i=1,2$. In this case we can use the following genus $2$ curve $Y$ with dominant maps $Y\to E_i$ considered in \cite[Section 4.2]{GL24}, first constructed in an unpublished work \cite{Scholten}. The assumption on the $2$-torsion implies that the curve $E_1$ is isomorphic over $k$ to an elliptic curve with equation, $E_1 :y_1^2=x_1(x_1-a)(x_1-b)$, and similarly $E_2 : y_2^2=x_2(x_2-c)(x_2-d)$. The genus 2 curve $Y$ with affine equation 
    \[Y: (ad-bc)y^2=((a-b)x^2-(c-d))(ax^2-c)(bx^2-d)\] 
    covers $E_1$ by mapping $(x, y)$ to $\left(\frac{ad - bc}{(a - b)ab}x^2 + \frac{c - d}{a - b}, \frac{ab y}{ad - bc}\right)$, and covers $E_2$ similarly. As $Y$ covers both curves, its Jacobian is isogenous to the product of $E_1\times E_2$, and so $Y$ maps birationally into $E_1 \times E_2$. When $ad-bc=0$, the curve $Y$ is degenerate, which can be remedied by 
    changing $E_1$ or $E_2$ by an isomorphic curve (see \cite[(4.1.2)]{GL24}). The argument can be repeated as in the previous case, with the curve $Y$ in place of $X$.
\end{proof}

 \begin{remark}
    The proof of Theorem~\ref{Ec w/o bad j-invariant} can be rephrased using the language of Kummer surfaces. The Kummer surface $K=\Kum(E_1\times E_2)$ attached to the product of the two elliptic curves is the K3 surface obtained by the minimal resolution of singularities of the quotient $(E_1 \times E_2)/\langle -1 \rangle$. If the elliptic curves are given by equations $E_i:y_i^2=f_i(x_i)$, for $i=1,2$, then the surface $K$ has an affine chart given by the equation $y^2=f_1(x_1)f_2(x_2)$, and there is a rational map $E_1\times E_2\dashrightarrow K$, given by $(x_1,x_2,y_1,y_2)\mapsto(x_1,x_2,y_1y_2)$.
     The curves $X$, and $Y$ considered in the proof of the above theorem and their iterates $[n]_\star X, [n]_\star Y$  descend to rational curves in the Kummer surface $K=\Kum(E_1\times E_2)$.  In particular, $K$ contains infinitely many rational curves yielding  $\dendegs(K/k)=\nn$ by Lemma~\ref{lem:covcurvs}.
 \end{remark}

\begin{coroll} In both cases considered in Theorem~\ref{Ec w/o bad j-invariant} 
 there exists a quadratic field over which both elliptic curves have an increase in rank unconditionally on the ranks of $E_1(k), E_2(k)$.
    \label{quadrankjump}
\end{coroll}
\begin{proof}
    In the proof of the previous theorem, we constructed hyperelliptic curves inside of $E_1 \times E_2$, and in particular, the $x$-coordinates of $E_1$ and $E_2$ can be expressed in terms of the hyperelliptic map. This implies the quadratic points on $E_1$ and $E_2$ corresponding to the pull-backs of rational points through these hyperelliptic maps have rational $x$-coordinates, and so they are trace zero. Arguing as in \Cref{lem:rankgoesup}, gives the result. 
\end{proof}

\begin{prop}\label{prop:parity}  Under the Parity Conjecture, $2 \in \dendegs(E_1 \times E_2 / k)$ for any elliptic curves $E_1$ and $E_2$ over a number field $k$ with a real embedding.
\end{prop}

\begin{proof} The only case left in question is when $\rk(E_1(k))=\rk(E_2(k))=0$ and both elliptic curves have $j$-invariant $0$ or $1728$. By Lemma~\ref{lem:deltaE}, it suffices to find a quadratic extension $L/k$ over which both curves have positive rank. Let $L/k$ be a quadratic extension in which all places of bad reduction of $E_1$ and $E_2$ are split, and exactly one real place of $k$ is not split; such an extension exists since it is specified by finitely many local conditions. Over $L$, both the root numbers of $E_1$ and $E_2$ are negative. Indeed, the product of the local root numbers for the finite places is $1$ since $W(E/K_{\frak{p}_1}) = W(E/K_{\frak{p}_2})$ for each place $\frak{p}_1, \frak{p}_2$ lying above a place of bad reduction $\frak{p}$, and there are an odd number of archimedean places by construction. Thus, by the Parity Conjecture, both $E_1$ and $E_2$ have positive rank over $L$.
\end{proof}

\begin{example} Let $k =\mathbb{Q}$ and let $E_1$ be the curve with LMFDB label \href{https://www.lmfdb.org/EllipticCurve/Q/3872/f/4}{3872.f4} defined by $y^2=x^3 + 484x$ and let $E_2$ be the curve with label \href{https://www.lmfdb.org/EllipticCurve/Q/16928/c/1}{16928.c1} defined by $y^2=x^3-92x$. Both curves have $j$-invariant $1728$ and rank $0$ over $k$. The primes of bad reduction are $2,23$ and $2,11$ for $E_1$ and $E_2$, respectively. Over $L = \mathbb{Q}(\sqrt{-7})$, these primes split, and both $E_1$ and $E_2$ have negative root number and positive rank.
\end{example}

One way to avoid using the Parity Conjecture when $j(E_1)=j(E_2)=0$ or $1728$ is to replace, if possible, one or both elliptic curves with isogenous ones that have different $j$-invariant. Namely,
assume that \(\rk E_i(k) = 0\) and that  there exists an isogenous curve \(E_1'\) with \(j(E_1')\neq j(E_1)\).  Then it follows by Theorem~\ref{Ec w/o bad j-invariant} and  Lemma~\ref{quadrankjump}  that $2\in\delta(E_1'\times E_2)$, and in fact there exists a quadratic extension \(L/k\) over which both \(\rk E_1'(L) >0\) and \(\rk E_2(L) >0\).  Since the rank is an isogeny invariant, we see that \(\rk E_1(L) >0\) and so \(2 \in \dendegs(E_1 \times E_2/k)\) by Lemma~\ref{lem:deltaE}.

\begin{example}
    Consider the elliptic curves 
    $$E_1: y^2=x^3+1, \quad E_2: y^2=x^3-8,$$ which have Mordell Weil groups $\bZ/6\bZ$ and $\bZ/2\bZ$ respectively, and $j_1=j_2=0$. The curve $E_1$ lies in the isogeny class with LMFDB label  \href{https://www.lmfdb.org/EllipticCurve/Q/144/a/}{36.a}, while $E_2$ in the class \href{https://www.lmfdb.org/EllipticCurve/Q/144/a/}{576.f}. In particular, $E_1, E_2$ are not isogenous over $\bQ$, and they don't have fully rational $2$-torsion. Thus, none of the earlier unconditional results apply in this case. However, the curve $E_1$ is isogenous to the curve $E_1': y^2=x^3-135x-594$, and $E_2$ is isogenous to $E_2':y^2=x^3-540x+4752$ with  $j_1'=j_2'=54000$. It follows by Theorem~\ref{Ec w/o bad j-invariant} that $2\in\dendegs(E_1'\times E_2'/k)$, which yields $2\in\dendegs(E_1\times E_2/k)$. 
Similarly, consider the elliptic curves $$E_1: y^2=x^3+4x, \quad E_2: y^2=x^3-289x,$$ which have Mordell Weil groups $\bZ/4\bZ$ and $\bZ/2\bZ\oplus\bZ/2\bZ$ respectively, and $j_1=j_2=1728$. The curve $E_1$ lies in the isogeny class with LMFDB label  \href{https://www.lmfdb.org/EllipticCurve/Q/144/a/}{32.a}, 
    and it is isogenous to the curve $E_1':y^2=x^3-11x-14$. 
    The curve $E_2$ lies in the isogeny class  \href{https://www.lmfdb.org/EllipticCurve/Q/144/a/}{9248.f}, 
    and it is isogenous to the curve $E_2':y^2=x^3-3179x-68782$. Both $E_1', E_2'$ have $j$-invariant $j=287496$. 
\end{example}

\subsection{Proof of \Cref{thm:preciseresults}\eqref{main:g1}} \label{subsec:proofofmain1}
We next consider all degrees $d\geq 3$. Applying Corollary~\ref{coroll:inclusionofN(C,D)}, we find $\nn_{\geq 10} \subset \dendegs(E_1 \times E_2 / k)$. Moreover, it follows by \Cref{inclusionsprod} that every composite number is contained in $\dendegs(E_1\times E_2/k)$. Thus, the only degrees left to check are 3, 5, and 7, which can be handled as one.

\begin{thm}\label{Ec:3,5,7}
    For any pair of elliptic curves, $E_1$, $E_2$, we have $3, 5, 7 \in \dendegs(E_1 \times E_2/k)$. 
    Hence $\mathbb{N}_{\geq 3}\subseteq \dendegs(E_1\times E_2/k)$.
    Moreover, for all \(d \geq 6\), the surface \(E_1 \times E_2\) admits dense degree \(d\) points with Galois group (of the Galois closure) \(S_d\).
\end{thm}
\begin{proof}
    As in our general strategy, we will construct a family of curves \(X_\lambda \to E_1 \times E_2\) of geometric genus \(4\) such that
    \begin{enumerate}
        \item \(\#X_\lambda(k) \geq 2\), and
        \item \(\bigcup_\lambda X_\lambda\) is Zariski dense in \(E_1 \times E_2\), and
        \item $X_\lambda$ admits degree $3, 5, 7$ maps to $\mathbb{P}^1$.
    \end{enumerate}
    Since a curve of genus \(4\) cannot have a degree \(2\) and a degree \(3\) map to \(\P^1\) by the Castelnuovo--Severi inequality, \(X_{\lambda}\) is not hyperelliptic. The statement about \(d \geq 6\) follows by Lemma~\ref{lem:Sd_points} and Remark~\ref{rem:asymptotic_Sd} (for \(d =7\), we use that a degree \(1\) divisor of the form \(2P_1 - P_2\) is not effective for points \(P_1 \neq P_2\) on a nonhyperelliptic curve.)

    Write $E_1$ as $y_1^2 = x_1^3 + a_1 x_1 + b_1$, and $E_2$ as $y_2^2 = x_2^3 + a_2 x_2 + b_2$. Let $X_{\lambda}$ denote the normalization of the fiber product of $E_1$ and $E_2$, over $(x_1, \lambda^2 x_2)$, for $\lambda \in k^{\times}$. We omit any $\lambda$ where $X_{\lambda}$ is reducible. Since both $E_1$ and $E_2$ are ramified over infinity, and $X_{\lambda}$ is biquadratic, it has two \(k\)-points \(\infty_1, \infty_2\) over infinity. By the Riemann--Hurwitz formula, the genus of \(X_\lambda\) is 4.

    We will show that each \(X_\lambda\) admits maps of degrees \(3, 5\) and \(7\) to \(\P^1\).  We begin with the degree \(3\) case.  Consider the meromorphic function $f_{\lambda} = y_1 - \lambda^3 y_2$, which determines a map to \(\P^1\). 
    Since 
    \begin{equation}\label{eq:diff_sqs}
        (y_1 - \lambda^3 y_2)(y_1 + \lambda^3 y_2) = y_1^2 - \lambda^6 y_2^2 = \lambda^2 (a_1 - a_2\lambda^4)x_2 + (b_1 - \lambda^6 b_2)
    \end{equation}
    is a linear function in \(x_2\), for any fixed nonzero value of \(f_\lambda\), we have that
    \(y_1 = f_\lambda + (y_1 + \lambda^3y_2)\) is a linear function in \(x_2\).  We also have that \(x_1 = \lambda^2x_2\) and \(y_2 = \lambda^{-3}(y_1- f_\lambda )\) are linear functions of \(x_2\).  Hence the fiber of \(f_\lambda\) over a generic point has length \(3\), corresponding to the roots of \(y_1^2 = x_1^3 + a_1 x_1 + b_1\), which is a cubic expression in \(x_2\).

    The curves $X_{\lambda}^\nu$ are not hyperelliptic, since by the Castelnuovo--Severi inequality, the bielliptic and hyperelliptic maps would have to factor through a common curve. We can therefore use \Cref{curve_asymptotic} to deduce $7 \in \dendegs(X_{\lambda}^\nu / k)$.
    Since the curves have at least two $k$-points ($\infty_1$ and $\infty_2$), by the same lemma we conclude that $5  \in \dendegs(X_{\lambda}^\nu / k)$ as well.
\end{proof}

Combining the proof of this result with Lemmas~\ref{lem:Sd_points} and \ref{lem:rankgoesup}, gives the following result.

\begin{coroll}\label{cor:simultaneousrankjumps}
    For any elliptic curves, $E_1$, $E_2$, defined over a number field, $k$, and $d \geq 3$, there exists an extension, $L$, of $k$ of degree $d$ such that $\rk{E_i(L)} \geq \rk{E_i(k)} + 1$ for $i=1,2$. Moreover, if $d \geq 6$, the Galois group of $L / k$ can be assumed to be $S_d$.
\end{coroll}
\begin{proof}
    By the proof of Theorem~\ref{Ec:3,5,7}, there exists a geometric genus 4 curve, $X$, inside $E_1 \times E_2$ such that for all $d \in \{3\} \cup \mathbb{N}_{\geq 5}$, $X$ has a $\mathbb{P}^1$-parameterized point \(x =(P_1,P_2)\) of degree $d$. 
    
    Furthermore, by the construction in the proof of Theorem~\ref{Ec:3,5,7}, if \(d \neq 6\), we may assume that \(x\) is linearly equivalent to a divisor supported over the origin in both \(E_1\) and \(E_2\) and if \(d \geq 7\) that the residue field \(k(x)\) is an \(S_d\)-extension; if \(d=6\) we may assume that \(x\) is linearly equivalent to \(K_X\) and \(k(x)\) is an \(S_6\)-extension.  Thus, using that \(3\) and \(5\) are prime, for all $d \in \{3\} \cup \mathbb{N}_{\geq 5}$
    the extension $k(x)$ is therefore primitive. As $k(x)$ is primitive, $k(P_1)$ and $k(P_2)$ are either $k$ or $k(x)$, and at least one is $k(x)$ as $(P_1, P_2)$ is a degree $d > 1$ point. If one of the $P_i$ is defined over $k$, its $x$-coordinate is, and so both $x$ coordinates are, this forces $k(x)$ to be degree 1 or 2, but $d \geq 3$. By the proof of the preceding theorem, the degree $d$ point, $x$, is either linearly equivalent to $K_X$ (in the case $d = 6$), or to a divisor, $D$, supported on the points at infinity. The push-forward of $D$ to $E_i$ is supported on the origin of $E_i$. Similarly, the push-forward of $K_X$ is linearly equivalent to a divisor supported on the origin, as the branch points of $X \to E_i$ are given by pairs of points with the same $x$-coordinate. In particular, the push-forward of $x$ is linearly equivalent to the origin, and so, as the addition laws on the $E_i$ respect linear equivalence, the $P_i$ are trace-zero. By the Northcott property, there are finitely many $x$ such that the $P_i$ are torsion, and the result follows by identical arguments as \Cref{lem:rankgoesup}. 

    For the case $d = 4$, as in the proof of  \Cref{lem:rankgoesup}, the degree $2$ maps $E_i\to\mathbb{P}^1$ induced by the complete linear systems $2\mathcal{O}_i$ give that there exist infinitely many degree 2 fields for which each $E_i$ gains rank. 
    Thus, there is a choice $(P_1,P_2)\in E_1\times E_2$ such that each $P_i$ is nontorsion degree $2$ and $k(P_1)\neq k(P_2)$.
    Then the point $(P_1, P_2)$ is defined over a biquadratic extension and so has degree 4.
\end{proof}

The following is the analog of Corollary~\ref{cor:k(mx)=k(x)} for a product of two elliptic curves.
\begin{coroll}\label{cor:k(P1,P2)=k(nP1,nP2)} For any elliptic curves $E_1, E_2$ over a number field $k$ and for every $d\in \{3\}\cup\nn_{\geq 5}$ there exists a degree $d$ point $(P_1,P_2)\in E_1\times E_2$ such that each $P_i$ is nontorsion, of degree $d$, and for every $n\geq 1$, $k(P_i)=k([n]P_i)$. 
    \end{coroll} 

\begin{remark}
If \(E_i[2] \subset E_i(k)\) for \(i=1,2\), the curves \(Y\) used in the proof of Theorem~\ref{Ec w/o bad j-invariant} can be used to show that in fact there exist infinitely many quadratic extensions over which both \(E_1\) and \(E_2\) increase rank.
\end{remark}

\subsection{Application: Density degree set of some Bielliptic Surfaces} \label{dendegsonbielliptic}
In this subsection we consider a surface $X$ which is the quotient $(E_1\times E_2)/G$ of a product of two elliptic curves modulo a fixed-point-free action of a finite group $G$ acting on $E_1$ by translations and on $E_2$ by $k$-automorphisms such that $E_2/G\simeq\mathbb{P}^1$. 
We assume that the action on $E_1$ is induced by a $k$-rational torsion point $P_0$. 
Such a surface is called \defi{bielliptic}. It has Kodaira dimension $0$, it admits a finite \'{e}tale cover $E_1\times E_2\xrightarrow{\pi} X$ of degree $n=|G|$, and it is not an abelian surface. 
The Albanese variety $\Alb_X$ of $X$ is the elliptic curve $E_1'=E_1/G=E_1/\langle P_0\rangle$, and the Albanese morphism $X\xrightarrow{\alb_X}E_1'$ is an elliptic fibration.  
These surfaces have been classified into $7$ different types (see for example \cite{Cil}) and we may assume that $P_0\in E_1(k)$ is a nonzero torsion point of $E_1$ of order at most $6$.
Corollary~\ref{cor:simultaneousrankjumps} allows us to obtain explicit information on the density degree set of such a bielliptic surface $X$.

\begin{coroll}\label{cor:biellipticsurface}
    Let $X=(E_1\times E_2)/G$ be a bielliptic surface over $k$ as above. Then we have containments \[\nn_{\geq 3}\subseteq\dendegs(X/k)\subseteq\dendegs(E_1/k).\] 
    Moreover, if any of the assumptions of Theorem~\ref{Ec w/o bad j-invariant} is true, then $2\in\dendegs(X/k)$ as well.
\end{coroll}
\begin{proof} The Albanese morphism $\alb_X:X\to E_1'$ is surjective, and hence 
\Cref{bound1} gives a containment $\dendegs(X/k)\subseteq\dendegs(E_1'/k)$. Since isogenous elliptic curves have the same density set, the upper bound $\dendegs(X/k)\subseteq\dendegs(E_1/k)$ follows. 

Let $d\in \{3\}\cup\nn_{\geq 5}$. Corollary~\ref{cor:k(P1,P2)=k(nP1,nP2)} gives that there exists a degree $d$ point $x=(P_1,P_2)\in E_1\times E_2$ such that $k(P_1)=k(P_2)$ and for all $n\geq 1$, $k(P_i)=k([n]P_i)$. 
This implies that the image $\phi(P_1)$ under the isogeny $\phi:E_1\to E_1'=E_1/\langle P_0\rangle$ has also degree $d$. Since the composition $E_1\times E_2\xrightarrow{\pi}X\xrightarrow{\alb_X}E_1'$ maps $(P_1,P_2)$ to $\phi(P_1)$, which is a degree $d$ point, it follows that $k(\pi(P_1,P_2))=k(P_1)=k(P_2)$ is also of degree $d$. 

Next suppose $d=4$. The above argument does not necessarily work in this case, as the degree $4$ points we constructed in the proof of \Cref{cor:simultaneousrankjumps} are of the form $(P_1,P_2)$ with $\deg(k(P_i))=2$ for $i=1,2$, and hence the above argument gives that $\pi(P_1,P_2)$ has degree $2$ or $4$. Suppose for contradiction that this degree is $2$, which can only happen if 
$k(\pi(P_1,P_2))=k(\phi(P_1))=k(P_1)$. Let $L=k(P_1, P_2)$, which is a degree $2$ extension of $k(P_1)$. Let $\sigma$ be a generator of the Galois group of $L/k(P_1)$. Then $\sigma(\pi(P_1,P_2))=\pi(P_1,P_2)=\pi(P_1, \sigma(P_2))$. By assumption $\sigma(P_2)\neq P_2$, thus this last relation gives a contradiction, as the group $G$ acts on $E_1$ by translations.

If the $j$-invariants of $E_1$ and $E_2$ are not both 0 or 1728, or the $2$-torsion subgroup of $E_1\times E_2$ is fully rational, then by \Cref{quadrankjump}, the same argument shows $2 \in \delta(X/k)$. 
\end{proof} 

If $1\in\dendegs(E_1\times E_2/k)$, then clearly $1\in\dendegs(X/k)$, and hence  $\dendegs(E_1\times E_2/k)=\dendegs(X/k)$ in this case. This equality also holds when $\rk(E_1(k))=0$. The only remaining question is to explore density of rational points on the bielliptic surface $X$ when the Albanese variety of $X$ has positive rank. The next proposition shows that the answer can sometimes be negative.

\begin{prop}\label{prop:strictbiellipticbound}
    There are examples of bielliptic surfaces $X$ for which  $1\not\in\dendegs(X/k)$, although $\rk(E_1(k))>0$. 
\end{prop}
\begin{proof}
Let $E_1, E_2$ be elliptic curves over $k$ with $\rk(E_1(k))>0, 
 \rk(E_2(k))=0$, and assume that there is a $k$-rational $2$-torsion point $P_0\in E_1(k)$. Consider the bielliptic surface $X=(E_1\times E_2)/G$, with $G=\bZ/2\bZ=\langle\sigma\rangle$ acting on $E_1$ by translation by $P_0$ and on $E_2$ by negation. The projection $\pi:E_1\times E_2\to X$ is a finite map of degree $2$, and hence a rational point on $X$ corresponds to a point $(P,Q)\in E_1\times E_2$ defined at worst over a quadratic extension $L/k$. Since $\rk(E_2(k))=0$, it follows that $1\in\dendegs(X/k)$ is only possible if there exists a quadratic point $(P, Q)\in E_1\times E_2$ with $Q$ of infinite order such that $\sigma(P, Q) = (\tau(P), \tau(Q))$, where $\tau$ denotes the non-trivial element of the Galois group. In particular, $\tau(Q) = -Q$ and $\tau(P)= P + P_0$, and hence $2P = P+\tau(P)- P_0\in E_1(k)$.

We conclude that $1\in\dendegs(X/k)$ can only happen if there exists a quadratic extension $L/k$ such that $\rk(E_2(L))>0$, and there exists $P \in E_1(L)$ such that $2P \in E_1(k)$. Since $E_1(k)$ is finitely generated, there are only finitely many such fields $L$. Thus, if $\rk(E_2(L))=0$ for each one of them, it follows that $X$ does not have dense rational points. 
\end{proof}

\begin{example}\label{ex:strictbielliptic}
Let $E_1$ be the elliptic curve \href{https://www.lmfdb.org/EllipticCurve/Q/65/a/1}{65.a.1} with Weierstrass equation $y^2 + xy = x^3 - x$. This curve has rank $1$ over $\bQ$, with a free generator the point $P=(1,0)$, and a nontrivial $2$-torsion point $P_0=(0,0)$. There are exactly $4$ quadratic fields over which $E_1$ has a point $P'$ with $2P' \in E_1(\bQ)$, namely $\bQ(\sqrt{65}), \bQ(\sqrt{5}), \bQ(\sqrt{13})$ and $\bQ(\sqrt{-3})$, where $2P'$ can be taken respectively to be $2P, P+P_0, P, P+P_0$. Let $E_2$ be the elliptic curve \href{https://www.lmfdb.org/EllipticCurve/Q/14/a/5}{14.a.5} with Weierstrass equation $y^2 + xy + y = x^3 - x$, which has rank $0$ over $\bQ$. A Magma computation shows that $E_2$ has rank $0$ over all $4$ quadratic extensions, yielding $1\not\in\dendegs(X/\bQ)$. 
However, one can take $E_2'$ to be the elliptic curve $y^2=x^3+4$, which has rank $0$ over $\bQ$, but has a point $(1,\sqrt{5})$ over $\bQ(\sqrt{5})$ of infinite order, thus yielding $1\in \dendegs(X/\bQ)$. 
\end{example}

\section{Products of an elliptic curve and genus \(2\) curve}
In this section we consider a product $E\times C$ of an elliptic curve $E$ and a nice genus $2$ curve $C$. If $E$ has positive rank, the lower bound of Proposition~\ref{inclusionsprod} coincides with the upper bound, hence $\dendegs( C\times E/k)=\dendegs(C/k)$. 
This leaves the rank zero case, in which case the two bounds do not coincide. We start with the following easy  Corollary. 

\begin{coroll}
    Let $E$ be a rank $0$ elliptic curve.
    If $\ind(C/k)=2$, then
    \[
        2\nn \setminus \{2\} \subseteq \dendegs(E\times C/k) \subseteq 2\nn.
    \]
\end{coroll}

\begin{proof}
    Since $\dendegs(E/k)=\nn_{\geq 2}$, then Proposition~\ref{inclusionsprod}, and Lemma~\ref{lemma:deltagenus2} give the above inclusions.
\end{proof}

Therefore, if $\ind(C/k)=2$ the only integer left to study is 2.

\subsection{Density of Quadratic points}\label{sec:quadraticpoints} Similarly to the case of a product of elliptic curves, the study of quadratic points on $E\times C$ can be related to the study of rational points on the quotient surface $(E\times C)/\langle\iota\rangle$ by the diagonal $\bZ/2$-action on $E\times C$ given by negation on $E$ and the hyperelliptic involution  on $C$.

The following conjecture (explicitly stated in \cite[Conj 1.1]{DeJard} for \(k = \bQ\), and implicitly in \cite{Maz92}, \cite{CCH05}) on elliptically fibered surfaces can be helpful for understanding quadratic points on these products.
\begin{conj}\label{conj:isotrivial_fibration}
    Let \(k\) be a number field. The \(k\)-points are dense on any non-trivial isotrivial elliptic fibration over $\mathbb{P}^1$.
\end{conj}

\begin{prop}
Assuming Conjecture~\ref{conj:isotrivial_fibration}, we have $2 \in \dendegs(E \times C/k)$. 
\end{prop}

\begin{proof}
    Write $E$ as the projective closure of $y_1^2=f(x_1) = x_1^3 + a_1x_1 + a_0$, and consider the affine chart of $C$ defined by $y_2^2=g(x_2)$ with $g$ a degree $6$ polynomial.
    
    Consider the quotient surface $\mathcal{E}=(E\times C)/\langle\iota\rangle$. This surface has an affine chart given by $y^2 = f(x_1)g(x_2)$. There is a degree \(2\) map $\phi \colon E \times C \to \mathcal{E}$ via $(x_1, y_1, x_2, y_2) \mapsto (x_1, x_2, y_1 y_2)$. 
    
    Since \(1 \not\in \dendegs(C/k)\), we have $1 \notin \dendegs(E \times C / k)$ by Corollary~\ref{cor:intersection}.  Hence, outside of a Zariski closed subset of \(\mathcal{E}\), the preimage of a rational point must be a quadratic point on \(E \times C\).  Thus if $1 \in \dendegs(\mathcal{E} / k)$, then $2 \in \dendegs(E \times C / k)$. 
    
    Via \((x_1, x_2, y) \mapsto x_2\), the surface $\mathcal{E}$ is an elliptic fibration over \(\P^1\).  This fibration is a family of quadratic twists of \(E\) given by \(g(x_2)\); hence it is isotrivial, but non-trivial. Thus Conjecture~\ref{conj:isotrivial_fibration} predicts that \(1 \in \dendegs(\mathcal{E}/k)\), and so \(2 \in \dendegs(E \times C/k)\).
\end{proof} 

It follows by Proposition~\ref{isogenyfactor} that in the special case when $E$ is an isogeny factor of $\Pic^0_C$, then $2$ is indeed in $\dendegs(E\times C/k)$. In the following lemma we use this fact to construct another class of examples that verify this conjecture.

\begin{prop}\label{lem:quad_24} Let $C$ be a genus $2$ curve given by an affine equation $y^2=g_4(x_2)g_2(x_2)$ where $g_4(x)$ is a degree 4 polynomial and $g_2(x)$ a monic degree 2 polynomial. Let $D$ be the normalization of the projective closure of $y_1^2=g_4(x_1)$. Let \(E \colonequals \Pic^0_D\), then $2\in \dendegs(E\times C/k)$.
 \end{prop}
\begin{proof}
     Let $X$ be the fiber product of $C$ and $D$ over \(\pp^1\) via the maps $x_1$ and $x_2$. This has equations $y_1^2 = g_4(x)g_2(x)$ and $y_2^2 = g_4(x)$. We can replace the first by $y_3^2 = g_2(x)$, where $y_1 = y_2y_3$. That is, we can write $X=\{(x,y_2,y_3): y_2^2 = g_4(x), y_3^2= g_2(x)\}$. Since $g_2$ is monic, the curve $y_3^2 = g_2(x)$ is a pointed conic, and so isomorphic to $\P^1$. Thus, the degree \(2\) map $X\to\mathbb{P}^1$, $(x,y_2,y_3)\mapsto (x,y_3)$  shows that $X$ is a hyperelliptic curve and $2 \in \dendegs(X / k)$. On the other hand, the degree \(2\) projection $X\to D$ implies that $E$ is an isogeny factor of $\Pic^0_X$, and so $2 \in \dendegs(X \times E / k)$ by Proposition~\ref{isogenyfactor}. The covering map $X \to C$ induces a covering $E \times X \to E \times C$. Since $1 \notin \dendegs(E \times C/k)$, the image of the degree 2 points on $E \times X$ must yield a Zariski dense set of degree 2 points on $E \times C$.
\end{proof}

\begin{example} 
Let $D$ be the curve with affine equation \[D:y^2 = g_4(x)=x^4 + 2x^3 - 5x^2 + 2x + 1.\] Then $D(\bQ)\neq\emptyset$ and \(D\simeq\Pic^0_D\) is the elliptic curve with LMFDB label \href{https://www.lmfdb.org/EllipticCurve/Q/14/a/5}{14.a.5}, which has rank 0. 
Let $C$ be the curve \[C:y^2=g_4(x)(x^2 + 2x - 3)=x^6 + 4x^5 - 4x^4 - 14x^3 + 20x^2 - 4x - 3.\] Then Lemma~\ref{lem:quad_24} gives $2\in\dendegs(D\times C/k)$. Notice that this example is not covered by Proposition~\ref{isogenyfactor} since $C$ has a geometrically simple Jacobian. 
\end{example}

\begin{coroll} Consider the set-up of Proposition~\ref{lem:quad_24}. Then the corresponding elliptic surface \(\mathcal{E}=(E\times C)/\langle \iota \rangle\) contains infinitely many rational curves, and hence we have an equality $\dendegs(\mathcal{E}/k)=\nn$. 
    \end{coroll} 
    \begin{proof}
         Let \(\mathcal{O}_D(1)\) denote the pullback of \(\mathcal{O}_{\pp^1}(1)\) via the degree \(2\) \(x\)-coordinate map and let \(i_D\) denote the corresponding involution.  Let \(i_C\) denote the hyperelliptic involution on \(C\).  There is a degree \(4\) map \(f \colon D \to E\) given by \(P \mapsto \mathcal{O}_D(1)(- 2P)\) and the involution \(i_D\) descends to the involution \([-1]\) on \(E\).  Thus we have a commutative diagram
    \begin{center}
        \begin{tikzcd}
            X \arrow[dr] \arrow[dd, bend right=10] \arrow[rr, bend left=10]  && \pp^1 \arrow[d, hook] \\
            & D \times C \arrow[d] \arrow[r] & D \times C/ \langle i_D \times i_C\rangle \arrow[d] \\
            X' \arrow[r, hook] & E \times C \arrow[r] \arrow[d, "{[n] \times \mathrm{id}}"] & E \times C / \langle \iota \rangle \arrow[d, "{[n]}"] \\
            & E \times C \arrow[r] & E \times C / \langle \iota \rangle \\
        \end{tikzcd}
    \end{center}
    where the map \([n] \times \mathrm{id}\) on \(E \times C\) descends to \([n]\) on the elliptic fibration \(E \times C / \langle \iota \rangle\) since \([n]\) and \([-1]\) commute.
    The image of \(X\) in \(D \times C/ \langle i_D \times i_C\rangle\) is the \(\pp^1\) with coordinates \((x, y_3)\) as observed in the proof of Lemma~\ref{lem:quad_24}.  Since the projection from $X'$ onto $E$
    is nonconstant, its images under \([n] \times \mathrm{id}\) for \(n \in \mathbb{Z}\) form an infinite family.  These give rise to an infinite Zariski dense collection of rational curves (with dense rational points) on the quotient \(E \times C/\langle \iota \rangle\).
    \end{proof}

\subsection{Density of large degree points}

In this section we prove that all degrees above an explicit (relatively small) bound are in the density degree set of the product of an elliptic curve and a genus \(2\) curve of index \(1\).  We split into two cases based on whether the genus \(2\) curve has a rational point.

\begin{prop}\label{prop:index1genus2andE}
    Let $E$ be an elliptic curve of rank zero and let \(C\) be a nice genus \(2\) curve with $C(k)=\varnothing$, but $\ind(C/k)=1$. 
    Then
    \[
    \nn_{\geq 2}\setminus \{2,3,5,7,11,13\} \subseteq  \dendegs(E\times C/k)\subseteq \nn_{\geq 2}.
    \]
\end{prop}

\begin{proof}
By Lemma~\ref{lemma:deltagenus2}, $3\in \dendegs(C/k)$, and more precisely it is in $\delta_{\mathbb{P}^1_k}(C/k)$.
 Therefore, by Proposition~\ref{inclusionsprod} we have that
    \[
        \nn_{\geq 2} \cdot  \nn_{\geq 2} = \nn_{\geq 2} \setminus \{ \textup{primes} \} \subseteq \dendegs(E\times C/k)\subseteq \mathbb{N}_{\geq 2}.
    \]  
    Applying  Corollary~\ref{cor:pointedcurvexind1} with the degree \(3\) map \(C \to \pp^1\) and the degree \(2\) map \(E \to \pp^1\) yields $\nn_{\geq 18}\subset \dendegs(E\times C/k)$. 
    As $C(k) = \emptyset$, the same is true of the genus 9 covering curves constructed in \Cref{cor:pointedcurvexind1}, and since they have degree 3 maps to elliptic curves, they cannot be hyperelliptic by the Castelnuovo--Severi inequality. \Cref{curve_asymptotic} then shows that 17 is in the density set of this family.
\end{proof}

\begin{thm}\label{thm:nngeq9indeltaCxE}
    Let \(E\) be an elliptic curve of rank \(0\) and let \(C\) be a nice genus \(2\) curve with \(C(k) \neq \emptyset\).
    Then $\nn_{\geq 2}\setminus \{2,3,5,7\} \subseteq \dendegs(E\times C/k)$.
    If \(C\) has a $k$-rational Weierstrass point, then $7\in \dendegs(E\times C/k)$ as well.
\end{thm}

\begin{proof}
    Since both \(E(k) \neq \emptyset\) and \(C(k) \neq \emptyset\), the product \(E \times C\) has a $k$-rational point as well.
    By Corollary~\ref{coroll:inclusionofN(C,D)}, any integer $d$ greater or equal than 
    \[
        N(E,C)=2(1+2-2+4-2+4)=14
    \]
    is in $\dendegs(E\times C/k)$.
    By Proposition~\ref{inclusionsprod} we have that
    \[
        \nn_{\geq 2} \cdot ( \{2\} \sqcup \nn_{\geq 4}) = \nn_{\geq 2} \setminus \{9, \textup{primes} \} \subseteq \dendegs(C\times E/k).
    \]
    Therefore it remains to show that $9$, $11$ and $13$ are in $\dendegs(E\times C/k)$, and $7 \in \dendegs(E\times C/k)$ if \(C\) has a rational Weierstrass point.

    We employ the strategy of Lemma~\ref{lem:fiber_product_construction} to construct a family of low genus curves covering \(E \times C\).
    We start with the Weierstrass case. Let $C$ be given by $y_1^2 = x_1^5 + a_3 x_1^3 + a_2 x_1^2 + a_1 x_1 + a_0$, and $E$ be given by $y_2^2 = x_2^3 + b_1 x_2 + b_0$. Consider the family of curves $X_t$ given by the fiber products of $C$ and $E$ over $x_2 = t^2 x_1$. As in  Lemmas~\ref{lem:fiber_product_construction} and \ref{lem:split_over_node}, $X_t^\nu$ is a curve of genus 6 with 2 rational points $\infty_{\pm}$ lying above the point at infinity on both \(C\) and \(E\). 
    This improved genus bound gives $13 \in \dendegs(E \times C / k)$.
    As $X_t^\nu$ is a double cover of an elliptic curve, it cannot also be hyperelliptic by the Castelnuovo--Severi inequality. Lemma~\ref{curve_asymptotic} implies that $9, 11 \in \dendegs(X_t^\nu / k)$, and so $9, 11 \in \dendegs(E \times C / k)$. 
    For 7, we use \Cref{cs_asymptotic}, and the remark immediately following it. As $X_t^{\nu}$ is a double cover of a genus 1 curve and it has a rational point which is not a branch point for this cover, for all $d < 6 - 2\times1 = 4$, $2\times6 - 2 - d$ is in $\dendegs(X_t^{\nu} / k)$. In particular, $7 \in \dendegs(X_t^{\nu} / k)$ and so in $\dendegs(E \times C / k)$.
    
    Assume now that the $k$-point is not Weierstrass. 
    As before, write $C$ as $y_1^2 = f(x_1)$, where $f$ is a monic degree 6 polynomial, and $E$ as $y_2^2 = g(x_2)$, where $g$ is a monic degree 3 polynomial. Let $X_t$ be the family of curves given by $x_2 = t^2 x_1$. These are now geometric genus 7 curves with 2 rational points. Similarly, they are not hyperelliptic, and so $13 \in \dendegs(X_t / k)$, and thus $13 \in \dendegs(C \times E / k)$. 

    Applying \Cref{cs_asymptotic} shows that, for all $d < 7 - 2 = 5$, we have $12 - d \in \dendegs(X_t^\nu / k)$. This gives 9 and 11.
\end{proof}

\subsection{Cubic points}

Under the Parity Conjecture, there are also examples with $3 \in \dendegs(E \times C/k)$.  We thank Samir Siksek for giving the idea behind this argument.

\begin{conj}[The Parity Conjecture]
Let \(E\) be an elliptic curve defined over a number field \(k\).  Then the global root number satisfies:
\[w(E/k) = (-1)^{\rk E(k)}.\]
\end{conj}

The Parity Conjecture implies that if \(w(E/k) = -1\), then the rank of \(E\) is odd, and hence positive.  Since the global root number is a product of local root numbers, it can be calculated more easily.  In particular, if \(E/k\) has semistable reduction, then
\begin{equation}\label{eq:Sven'sthm}
w(E/k) = (-1)^{m_k + u_k},
\end{equation}
where \(m_k\) is the number of places of split multiplicative reduction of \(E/k\) and \(u_k\) is the number of infinite places of \(k\)
\cite[Corollary 2.5]{DokchitserKellock}.

\begin{thm}\label{thm:parity_3}
    Assume the Parity Conjecture. 
    Let $E/\mathbb{Q}$ be an elliptic curve with split 
    semi-stable reduction.  Suppose that $C$ is a nice genus \(2\) curve with integral model $\mathcal{C}$ with good reduction at the bad primes of $E$. 
    Moreover, assume that 
    \begin{enumerate}
    \item for each finite prime $p$ of bad reduction of $E$, 
    \begin{center}
        $|\mathcal{C}_p(\mathbb{F}_{p})| < p + 1$ or $|\mathcal{C}_p(\mathbb{F}_{p})| \geq p + 10$,
    \end{center}
    \item there exists a rational map $f : \mathcal{C} \dashrightarrow \mathbb{P}^1_{\mathbb{Z}}$ over $\mathbb{Z}$, of degree 3, such that it is defined at each bad prime of $E$ and at least one $\mathbb{R}$-fiber is not totally split.
    \end{enumerate}
    Then $3 \in \dendegs(E \times C/\mathbb{Q})$.
\end{thm}

\begin{proof}
If the number of primes of bad reduction is even, then, since \(u_{\mathbb{Q}} = 1\), we have \(w(E/\mathbb{Q}) = -1\) and so \(\rk E(\mathbb{Q}) > 0\) and the result follows since \(3 \in \dendegs(C/\mathbb{Q})\). 
We therefore assume there are an odd number of such primes of bad reduction. 

Our approach will be to find points $t \in \mathbb{P}^1_\mathbb{Q}$ such that the fiber $f^{-1}(t)$ is a degree \(3\) point with residue field \(K\) and $E(K)$ has positive rank.

Let \(p\) be a prime of bad reduction of $E$.  By assumption, either \[|\mathcal{C}_p(\mathbb{F}_p)| < |\mathbb{P}^1(\mathbb{F}_p)|, \qquad \text{or} \qquad |\mathcal{C}_p(\mathbb{F}_p) > |\mathbb{P}^1(\mathbb{F}_p)| + 8 .\] 
In the first case, $f|_{\mathbb{F}_p}$ cannot be surjective on $\mathbb{F}_p$ rational points, and so there exists \(t_p \in \P^1(\mathbb{F}_p)\), such that $f^{-1}(t_p)$ is a degree 3 point over $\mathbb{F}_p$. 
In the second case, since $f$ has at most 8 \(\mathbb{F}_p\)-rational branch points by the Riemann-Hurwitz formula, there exists \(t_p \in \P^1(\mathbb{F}_p)\), outside of the branch locus, such that \(\#f^{-1}(t_p)(\mathbb{F}_p) \geq 2\). Since $f$ has degree 3, this implies \(\#f^{-1}(t_p)(\mathbb{F}_p) = 3\). In particular, if $t \in \mathbb{Q}$ reduces to $t_p$ modulo $p$, then in the first case, $p$ will be totally inert in $K$, and in the second, $p$ will be totally split at $K$.

We also have a fiber, $t_{\mathbb{R}}$, such that the fiber of $f$ above $t_{\mathbb{R}}$ consists of a real point and a complex-conjugate pair. For $t$ sufficiently near $t_{\mathbb{R}}$, this will imply $K$ has a real and complex embedding.
By weak approximation, there exist infinitely many $\mathbb{Q}$-rational points in $\mathbb{P}^1_\mathbb{Q}$ which restrict to $t_p$ for each bad prime $p$, and sufficiently close to $t_{\mathbb{R}}$.
For each such $t$, $E$ has positive rank over $K$ by \eqref{eq:Sven'sthm} and the Parity Conjecture, since $E / k$ has an odd number of finite places of bad reduction, and an even number of infinite places. In particular, $f^{-1}(t) \times E \subset C \times E$ has dense $K$-rational points. Repeating this for the infinitely many suitable $t$ gives the density of degree 3 points.
\end{proof}

\begin{example}
    Therefore, conjecturally, $3 \in \dendegs(X_1(11) \times X_1(18)/k)$.
    The modular curve $X_1(18)$ has a degree 3 morphism to $X_1(9)$ from the modular structure, and $X_1(9)$ is genus 0. The only bad prime of $X_1(11)$ is $11$. From an explicit model, modulo 11, $X_1(18)$ only has 9 points. The existence of a non-split real fiber comes from real elliptic curves with only one real component as there is a unique two-torsion point on that component. 
    
    We observe that this implies the existence of infinitely many cubic fields over which there are elliptic curves with an 11-torsion point, and elliptic curves with an 18-torsion point. These will not, in general, be the same elliptic curves, as $X_1(198)$ does not have infinitely many cubic points \cite[Theorem 3.1]{DKS04}. 
    Nor can the elliptic curves with 11-torsion be the known examples over $\mathbb{Q}$, since these points correspond to a closed locus in $X_1(11) \times X_1(18)$.
\end{example}

\section{Products of genus \(2\) curves} 
In this section, we obtain results for a product $C\times D$ of two nice genus $2$ curves over $k$. We start by summarizing the upper and lower bounds for $\dendegs(C\times D/k)$ which follow from the results of Sections~\ref{bounds} and ~\ref{sec: asymptoticresults}.

First consider the case when $\ind(C/k)=\ind(D/k)=1$. Recall from Lemma~\ref{lemma:deltagenus2} that for a genus $2$ curve $C$ with index $1$, we have that $3\in\dendegs(C/k)$ if and only if there exists $P\in C(\overline{k})$ of degree $3$. Moreover, if $3\not\in\dendegs(C/k)$, then $C$ necessarily has a rational point. When both curves have a degree $3$ point, Proposition~\ref{inclusionsprod} gives $\dendegs(C\times D/k)\subseteq\nn_{\geq 2}$. Otherwise, $\dendegs(C\times D/k)\subseteq\nn_{\geq 4}\cup\{2\}$. 

\begin{small}

\begingroup
\renewcommand*{\arraystretch}{1.1}
\begin{table}[h!]
\begin{tabular}{| c | c | c |c|} 
\hline
& 
$C(k)\neq\emptyset$, $3\in\delta_C $  
& 
$C(k)\neq\emptyset$, $3\not\in\delta_C $  
& 
$C(k)=\emptyset$, $3\in \delta_C$   
\\
\hline
$D(k)\neq\emptyset$, $3\in\delta_D$  
& 
$\nn\setminus \{\text{primes}\leq 11\}$ 
& 
$\nn\setminus\{9, \text{primes}\leq 11\}$ 
& 
$\nn\setminus \{\text{primes}\leq 17\}$ 
\\
\hline
$D(k)\neq\emptyset$, $3\not\in\delta_D$
& $\nn\setminus\{9,\text{primes}\leq 11\}$ 
& $\nn\setminus\{6, 9, \text{primes}\leq 11\}$ 
& $\nn\setminus\{9, \text{primes}\leq 17\}$ 
\\
\hline 
$D(k)=\emptyset$, $3\in \delta_D$
&  $\nn\setminus \{\text{primes}\leq 17\}$ 
& $\nn\setminus\{9,\text{primes}\leq 17\}$ 
& $\nn\setminus\{\text{primes}\leq 67\}$
\\
\hline
\end{tabular}
\caption{Table 1. Lower bound for $\dendegs(C\times D/k)$ when $\textup{ind}(C/k) = \textup{ind}(D/k)=1$}\label{Table1}
\end{table}
\endgroup
\end{small}

\begin{prop}\label{prop:summary1}
    Let $C, D$ be two nice curves of genus $2$ with $\ind(C/k)=\ind(D/k)=1$. 
    Then, according to the arithmetic of $C$ and $D$, the density set $\dendegs(C\times D/k)$ contains the sets listed in Table~\ref{Table1}. 
\end{prop}

\begin{proof}
First, consider the case when $C(k), D(k)\neq\emptyset$.
Corollary~\ref{coroll:inclusionofN(C,D)} constructs a family of geometric genus at most 9 curves in $C \times D$, which  yields an inclusion 
\[\nn_{\geq 18}\subseteq \dendegs(C\times D/k).\]
We remark that by the Castelnuovo--Severi inequality, these curves cannot be hyperelliptic as they double cover genus 2 curves.
We consider separately the cases where $C$ and $D$ have both rational Weierstrass points.
\begin{enumerate}
    \item[(i)] If both $C$ and $D$ have rational Weierstrass points, the covering family generically has a node, and the geometric genus is at most 8. 
    By applying a translation to one of these curves, we may assume that the covering curves are generically smooth away from this identified node, and so have genus exactly 8. By \Cref{lem:split_over_node}, for infinitely many members of this family, the normalization has two rational points over this node, $\infty_{\pm}$. In particular, the degree 2 covering map to $C$ is not branched at these points. 
    
    Since these covering curves are not hyperelliptic, \Cref{curve_asymptotic} implies that $\nn_{\geq 12}$ is contained within $\dendegs(C \times D / k)$.
    Applying \Cref{cs_asymptotic} and \Cref{rmk:applicationofLemma2.5} for $Z_1= d\infty_+$, for all $d < 8 - 2\times2 = 4$, the integer $2\times 8 - 2 - d = 14 - d$ is in the density degree set. Therefore, we also have $11\in \delta(C\times D/k)$.

    \item[(ii)] We now assume that either $C$ or $D$ has no rational Weierstrass point. Since the covering family has no fixed locus, by Bertini's Theorem, the generic element is smooth (of genus 9), and has at least two rational points at infinity. 
    For at least one of the covering maps to $C$ and $D$ infinity is not a branch point then, and there as before \Cref{cs_asymptotic} gives that, for all $d < 5$, the integer $16 - d$ is in the density degree set.
\end{enumerate}

We now divide into subcases according to whether $C$ or $D$ have degree 3 points.
\begin{enumerate}
    \item If both curves have a point of degree $3$, then Lemma~\ref{lemma:deltagenus2} and the lower bound of Proposition~\ref{inclusionsprod} give $\nn_{\geq 2}\cdot\nn_{\geq 2}=\nn_{\geq 2} \setminus\{\text{primes}\} \subset \dendegs(C\times D/k)$. 
    We conclude that 
    \[\nn_{\geq 2} \setminus \{p \text{ prime} : p\leq 11\}\subseteq\dendegs(C\times D/k).\]
    
    \item If exactly one of the two curves has a point of degree $3$, then $9\not\in\dendegs(C/k)\cdot\dendegs(D/k)$, which gives the $(1,2)$ and $(2,1)$ entries of Table~\ref{Table1}.

    \item When neither of the curves has a degree $3$ point, the lower bound is also missing 6, giving the $(2,2)$ entry of Table~\ref{Table1}.
\end{enumerate}

Next consider the case when exactly one of the curves has a rational point. Without loss of generality, we may assume that $C(k)=\emptyset$. 
Then $3\in\dendegs(C/k)$ by Lemma~\ref{lemma:deltagenus2}, and moreover there is a degree $3$ map $\pi_C:C\to\mathbb{P}_k^1$. Applying Corollary~\ref{cor:pointedcurvexind1} to $\pi_C$ and the standard degree $2$ cover $\pi_D:D\to \mathbb{P}_k^1$ gives a family of curves of genus 12 covering the surface, and yields an inclusion 
\[\nn_{\geq 24}\subset\dendegs(C\times D/k).\] 
Since $C$ has no points, the same is true of the members of the family, and they are not hyperelliptic, so by \Cref{curve_asymptotic}, 23 is in the density degree set of these curves. These curves also have a degree 3 divisor lying above the rational point of $D$, which satisfies the assumptions of \Cref{cs_asymptotic} applied to the double covering of $C$. This shows $22 - 3 = 19$ is in the density degree set.
As before, we split on whether $3$ is in the density set of $D$.
\begin{enumerate}
    \item If $3\in\dendegs(D/k)$, then the lower bound of Proposition~\ref{inclusionsprod} gives once again that every composite integer lies in $\dendegs(C\times D/k)$. We conclude that in this case \[\nn\setminus\{p \text{ prime}: p\leq 17\}\subseteq\dendegs(C\times D/k).\] 
    \item If $3\notin \dendegs(D/k)$, then $9\not\in\dendegs(C/k)\cdot\dendegs(D/k)$, thus $\nn\setminus\{9, \text{ primes}\leq 17\}\subseteq\dendegs(C\times D/k).$
\end{enumerate}

Lastly, we consider the case when $C(k)=D(k)=\emptyset$. Then $2,3\in\dendegs(C/k)\cap\dendegs(D/k)$, hence $C\times D$ is guaranteed to have zero-cycles of degrees $2\cdot 2=4$ and $3\cdot 3=9$. This gives an upper bound for the effective index (cf. Definition~\ref{def:effind}),
$e\leq 9+ 4=13$, and \Cref{thm:generalN} gives an inclusion 
$\nn_{\geq 512}\subset\dendegs(C\times D/k)$. This containment is far from sharp, however, as the proof of Theorem~\ref{thm:generalN} can be optimised in this case. Since the zero cycles used to define the effective index are products of divisors on the two curves, we can use these constituents to form a degree 5 divisor on each of the two curves. Since both curves are genus 2, these divisors are base-point free. Using these two maps gives a family of genus 36 (nonhyperelliptic, by Castelnuovo--Severi inequality) curves by \Cref{lem:fiber_product_construction}. These curves have no $k$-points, since neither $C$ nor $D$ do, and so \Cref{curve_asymptotic} shows that $71$ is in the density degree set, as well as $\nn_{\geq 72}$.
\end{proof}

Let us now briefly consider the setting when at least one of the curves has index $2$, in which case Proposition~\ref{inclusionsprod} gives an inclusion $\dendegs(C\times D/k)\subseteq 2\nn$.
\begin{lemma}
Let $C,D$ be two nice genus 2 curves with $\ind(C/k)=2$. 

\begin{enumerate}
    \item If $\ind(D/k)=1$, then $2\mathbb{N}\setminus \{2,6\} \subseteq \dendegs(C\times D/k)$.
    Moreover, if $3\in \dendegs(D/k)$ then $6 \in \dendegs(C \times D / k)$.
    \item If $\ind(D / k) = 2$ and $\ind(C \times D / k) = 2$, then $4\nn \cup 2\nn_{\geq (e + 3)^2} \subseteq \dendegs(C \times D / k)$, where $e$ is the effective index.
    \item If $\ind(D / k) = 2$ and $\ind(C \times D / k) = 4$, then $\dendegs(C \times D / k) = 4\nn.$
\end{enumerate}
\end{lemma}

\begin{proof}
We start with the case $\ind(D / k) = 1$. By \Cref{inclusionsprod}, $2\nn \cdot \dendegs(D / k) \subseteq \dendegs(C \times D / k)$. As the two possibilities for $\dendegs(D / k)$ are $\nn_{\geq 2}$ and $\nn_{\geq 2} \setminus \{3\}$, the result follows.

When $\ind(D / k) = 2$, and the product has index 2, the bound follows from \Cref{thm:generalN} as $N(C, D, e) = 2(e + 1)^2 + 8(e + 2) = 2(e + 3)^2$.

In the final case, the lower bound of \Cref{inclusionsprod} is $2 \nn \cdot 2 \nn = 4\nn$, which must be sharp, since the index of the surface is 4. 
\end{proof}

One case where the effective index is easy to calculate for a product of two genus 2 curves is when $C \times D$ has a degree 2 point, and we record this case separately. This occurs, for example, when $C$ and $D$ are isomorphic, which is the topic of the next subsection.

\begin{coroll}\label{cor:summary2}
    Let $C,D$ be two nice genus $2$ curves with $\ind(C/k)=\ind(D/k)=2$ and suppose that $C\times D$ has a degree $2$ point. Then 
    \[2\nn_{\geq 25}\subseteq \dendegs(C\times D/k).\]
\end{coroll}

\begin{proof}
In this case the index of $C\times D$ is achieved by a point $(P,Q)\in C\times D$ of degree $2$, and hence the effective index $e=2$. 
\end{proof}

\subsection{Pulling back points from the Jacobian} In this subsection we mainly focus on self-products $C\times C$ of a genus $2$ curve. We employ a source of obtaining dense degree $2d$ points on $C\times C$ by pulling back degree $d$ points on $\Pic^0_C$. 
\begin{thm}\label{thm:Jac of genus 2} 
Let $C$ be a nice genus $2$ curve over $k$. 
Then $\nn_{\geq 2}\subset\dendegs(\Pic^0_C/k)$. 
\end{thm}
\begin{proof}
    If $\Pic^0_C$ contains a curve $D$ of genus at least 2, then the images of the curve under $[n]$ are dense. 
    In particular, the density degree set of $\Pic^0_C$ contains the density degree set of $D$. 
    First, consider the map 
    \[
        C\xhookrightarrow{\Delta_C} C\times C \to \Sym^2_C \overset{AJ_C}{\to} \Pic^2_C \overset{\otimes \omega_C^{-1}}{\to} \Pic^0_C.
    \]
    The image of the composite is birational to $C$ (only the Weierstrass points get identified), and therefore $\dendegs(C/k)\subseteq \dendegs(\Pic^0_C/k)$. Since $C$ is hyperelliptic, this gives $2\nn\subset \dendegs(\Pic^0_C/k)$.
    
    We now show that there exists a curve $D$ in $\Pic^0_C$ of arithmetic genus 4 with dense degree 3 points, such that its normalization $D^\nu$ has a rational point (giving index 1).
    We may write $C : y^2 = g(x) = b_6 x^6 + b_4 x^4 + b_3 x^3 + b_2 x^2 + b_1 x + b_0$ with at least one of $b_3$ and $b_1$ non-zero. 
    If not, applying the M\"obius transformation $\frac{1}{tx + 1} - s$ gives another model, $C'$, for $C$, where $t \in k^\times$, and $s$ is chosen such that the new model has $b_5 = 0$. If for the new model $b_3 = b_1 = 0$ as well, there is an automorphism of $C'$ given by $(x, y) \mapsto (-x, y)$. Composing this with the change of model gives an automorphism of $C$, mapping the $x$-coordinate to $-\frac{(2st + 1)x + 2s}{(2st^2 + 2t)x + (2st + 1)}$. These automorphisms will be distinct for different $t$, which implies that for all but finitely many $t$, the new model will satisfy this assumption.
    
    Define $X_\sigma = C \times_{\sigma} C$, where $\mathbb{P}^1_k \xrightarrow{\sigma} \mathbb{P}^1_k $ is the automorphism sending $x\mapsto -x$ and $C\to \mathbb{P}^1_k$ is the usual $2:1$ cover. 
    Notice that the $S_2$-action on $C\times C$ (swapping the coordinates) restricts to the closed subscheme $X_\sigma$, which means that $X_\sigma$ descends to a closed subscheme $Z_\sigma$ in $\Sym^2_C$. 
    If $\{\infty_1,\infty_2\}\subseteq C$ are the two $\overline{k}$-points above $\infty\in \mathbb{P}^1_k$, then the orbit $\{(\infty_1),(\infty_2)\}$ descends to a unique $k$-point on $Z_\sigma$, and hence index$(Z_\sigma)=1$.
    Over the affine chart
    \[
    \Spec R = \Spec \left( k[x ,y_1,y_2]/ (y_1^2- g( x ), y_2^2- g(\sigma x )) \right)
    \]
    of $X_\sigma$, the $S_2$-action on $X_\sigma$ can be written as $(x, y_1, y_2) \mapsto ( \sigma x, y_2, y_1)$. 
    Consider the ring of invariants
    \[
    R^{S_2} = \left( k[x ,y_1,y_2]/ (y_1^2- g( x ), y_2^2- g(\sigma x )) \right)^{S_2},
    \]
    so that $\Spec(R^{S_2})$ is an affine chart of $Z_\sigma$.
    Notice that $u=x^2, w_1= y_1+y_2,  w_2 = y_1y_2$ are all invariants.
    Moreover we have the element $v = \frac{y_1-y_2}{x}\in R_x^{S_2}$ as well.
   In $R_x^{S_2}$ we can also write $w_2$ as $\frac{1}{4}\left(w_1^2 - uv^2\right)$.
    Let us check that $u,v,w_1$ generate the ring of invariants $R_x^{S_2}$. 
    Pick 
    \[
    p(x,y_1,y_2) = \sum_{i\in \mathbb{Z}} \alpha_i x^i + \sum_{i\in \mathbb{Z}} \beta_i x^iy_1 +\sum_{i\in \mathbb{Z}} \gamma_i x^i y_2  + \sum_{i\in \mathbb{Z}} \delta_i x^i y_1y_2 \in R_x^{S_2}
    \]
    and apply the $S_2$ action.
    We then get
    \[
    p(x,y_1,y_2) = \sum_{i\in \mathbb{Z}} (-1)^i \alpha_i x^i + \sum_{i\in \mathbb{Z}} (-1)^i \beta_i x^iy_2 +\sum_{i\in \mathbb{Z}} (-1)^i \gamma_i x^i y_1  + \sum_{i\in \mathbb{Z}} (-1)^i \delta_i x^i y_1y_2,
    \]
    which implies that $\delta_{odd}=\alpha_{odd}=0$ and that $(-1)^i\beta_i=\gamma_i$.
    This means that $p(x,y_1,y_2)$ can be written as a polynomial in $u,v,w_1$.
    Notice that these elements satisfy the following relations:
    \[
    vw_1 = \frac{(y_1-y_2)(y_1+y_2)}{x} = \frac{y_1^2 - y_2^2}{x}  = 2(b_3x^2+b_1) = 2(b_3u+b_1),
    \]
    \[
    \frac{1}{2}(w_1^2+ uv^2) = w_1^2 - 2w_2 = 2b_6 u^3 + 2b_4 u^2 + 2b_2 u + 2b_0.
    \]
    
    This realizes the open $\Spec( R_x^{S_2} )$ of $Z_\sigma$ as a closed of 
    \[
        U = \Spec \left( k[u,v,w_1]/ (vw_1-(2b_3u+2b_1), w_1^2+uv^2 - (4b_6u^3 + 4b_4u^2 + 4b_2u + 4b_0)) \right).
    \]
    Since $U$ is integral and $Z_\sigma$ has the same dimension as $U$, then $U=\Spec( R_x^{S_2})$.
    In particular $Z_\sigma$ is birational to the intersection of a quadric and a cubic and  $D=\overline{U}\to Z_\sigma$ is a birational morphism with $D$ a curve with arithmetic genus 4.
    If $b_3 \neq 0$, then $u = \frac{vw_1 - 2b_1}{2b_3}$ equips the quadric with a ruling over the base field, and thus $X_\sigma$ with degree 3 morphisms to $\mathbb{P}^1$ (namely $v$ and $w_1$). Otherwise, $vw_1 = b_1$, and eliminating $v$ from the second equation shows $b_1^2 u + w_1^4 = 2(b_6 u^3 + b_4 u^2 + b_2 u + b_0)w_1^2$. This also has a degree 3 morphism given by $w_1$. 

    Pushing forward $Z$ under the natural maps $\Sym^2_C \to \Pic^2_C \to \Pic^0_C$ gives a curve in $\Pic^0_C$. 
\end{proof}

\begin{prop}\label{CxC} Let $C$ be a nice genus $2$ curve over $k$. Let $\Pic^0_C$ be the Jacobian of $C$. Then  \[2 \dendegs(\Pic^0_C/k) \subseteq \dendegs(C \times C/k).\] In particular, if \(\Pic^0_C\) is simple and \(\rk \Pic^0_C(k) > 0\), then \(2 \in\dendegs(C \times C/k)\).

\end{prop}
\begin{proof} 
Since $g_C=2$, the Jacobian $\Pic^0_C$ is an abelian surface. 
Similarly to the proof of Theorem~\ref{thm:Jac of genus 2},
we consider the birational morphism 
\[
\Sym^2_C 
\xrightarrow{AJ_C} 
\Pic^2_C \xrightarrow{\otimes \omega_C^{-1}}  
\Pic^0_C.
\]
Precomposing with the canonical map $C\times C\xrightarrow{\pi}\Sym^2(C)$ 
gives a morphism $C\times C\xrightarrow{\sigma} \Pic^0_C$, and there exists a nonempty open subset $U\subset C\times C$ such that the restriction $U\xrightarrow{\sigma} \sigma(U)$ is a finite morphism of degree $2$. 
Let $d\in\dendegs(\Pic^0_C/k)$. 
Then  $d\in\dendegs(\sigma(U)/k)$, and every degree $d$ point of $\sigma(U)$ pulls back to a point of $U$ of degree either $d$ or $2d$. 

If a dense subset of the preimages have degree $2d$, we are done. Otherwise, we conclude that $d\in\dendegs(C\times C/k)$ and in particular $d\in\dendegs(C/k)$. Since $C$ is hyperelliptic, $2\in\dendegs(C/k)$, and hence the lower bound $\dendegs(C/k)\cdot\dendegs(C/k) \subseteq \dendegs(C\times C/k)$ guarantees that $2d\in\dendegs(C\times C/k)$. 

Now if $\Pic^0_C$ is simple and $\rk \Pic^0_C(k)>0$, it follows by Proposition~\ref{prop:simpleabeliansurface} that $1\in\dendegs(\Pic^0_C/k)$, and hence $2\in\dendegs(C\times C/k)$. 
\end{proof} 

We note that when $\ind(C/k)=2$ Proposition~\ref{CxC} improves a lot the lower bound obtained in Corollary~\ref{cor:summary2}. 
The following Corollary even shows that the upper bound of Proposition~\ref{inclusionsprod} can sometimes be achieved.

\begin{coroll}
    Let $C$ be a genus $2$ curve over $k$. Then $2\nn_{\geq 2}\subset \dendegs(C\times C/k)$. 
    If $C$ has index $2$ and $\Pic^0_C$ is a simple abelian surface with positive rank over $k$, then $\dendegs(C\times C/k)=2\nn$. 
    In particular, the upper bound of Proposition~\ref{inclusionsprod} is achieved in this case. 
\end{coroll}

\begin{proof}
    Theorem~\ref{thm:Jac of genus 2} and Proposition~\ref{CxC} give $2\nn_{\geq 2}\subset\dendegs(C\times C/k)$.  
    Moreover, if $\Pic^0_C$ is simple and $\rk(\Pic^0_C(k))>0$, then it follows by Proposition~\ref{prop:simpleabeliansurface} that $2\in\dendegs(C\times C/k)$. 
We conclude that $\dendegs(C\times C/k)=2\nn$, which is precisely the upper bound of Proposition~\ref{inclusionsprod}. 
\end{proof}

\begin{remark} 
In general, the main contribution of Proposition~\ref{CxC} to the $\ind(C/k)=1$ case seems to be that $2\in\dendegs(C\times C/k)$ when $\Pic_C^0$ is simple and has positive rank.
However, in certain cases we do get more information. 
Namely, suppose that $C(k)\neq\emptyset$, but $C$ has no degree $3$ points (e.g., \cite[Example 5.1.3]{VV}). 
Then Proposition~\ref{CxC} gives $6\in\dendegs(C\times C/k)$, something which was not guaranteed by \Cref{prop:summary1}. 
\end{remark}

\begin{example}
    The methods exhibited above for the self-product of a genus $2$ curve can be extended to products $C\times D$ with $g_D\geq 3$, in the special case when there is a finite cover $\phi: D\xrightarrow{}C$ of degree $m\geq 2$. For example, assuming that $\Pic_C^0$ is simple and has positive rank, we can deduce that $2m\in\dendegs(C\times D/k)$. 
    Indeed, Proposition~\ref{CxC} gives $2\in\dendegs(C\times C/k)$ and a general degree $2$ point on $C\times C$ pulls back either to a degree $2m$ point in $C\times D$ or, arguing similarly to the proof of \ref{CxC}, we would get that $m\in\dendegs(D/k)$
    in which case $2m\in\dendegs(C\times D/k)$ by the lower bound of Proposition~\ref{inclusionsprod}. 
    An explicit example can be found in \cite[Example 7.1]{bruin202322decomposablegenus4jacobians}, where  Nils Bruin and Avinash Kulkarni constructed (using theory they developed in \cite[Theorem 1.1, Proposition 4.4]{bruin202322decomposablegenus4jacobians}) a genus $4$ non-hyperelliptic curve $D$ whose Jacobian is isogenous to the square of the Jacobian of a genus 2 curve $C$. Moreover, there is a degree $2$ cover $D\to C$ and $\Pic^0_C$ is absolutely simple. Since $D$ is non-hyperelliptic and it does not admit a morphism to an elliptic curve, it follows by \cite[Lemma 3.2]{VogtKadets2024} that $2\not\in\dendegs(D/k)$. By extending the base field if necessary, we may assume that $\rk\Pic_C^0(k)>0$. Then our pullback method shows that $4\in\dendegs(C\times D/k)$, something which was not guaranteed by Proposition~\ref{inclusionsprod}.
\end{example}

\begin{example}\label{ex:isomorphicJacobians}
A similar method can be used for a special family of non-isomorphic genus $2$ curves $C$, $D$ with $C(k), D(k)\neq\emptyset$ to deduce that $3 \in \dendegs(C \times D / k)$. Namely, take $C$, $D$ as in \cite[Theorem 2]{howe05}, so that $\Pic^0_C \cong \Pic^0_D$ as unpolarised abelian varieties. The morphism \[C \times D \to \Pic^0_C \times \Pic^0_D \to (\Pic^0_C)^2 \to \Pic^0_C,\] given by the natural inclusions followed by the isomorphism of abelian varieties and summation, has degree equal to the intersection of the two curves in the Jacobian. By construction, these two curves differ only by the action of an automorphism on an isogenous abelian variety. Using standard formulae relating traces of endomorphisms to intersection products shows that in this case, the intersection number is 3. \end{example}

\subsection{Non-density of small degree points}\label{sec:nondensity} 
The previous two sections suggest that at least conjecturally we expect quadratic points to be dense on a product $C\times D$ when at least one factor is an elliptic curve. Moreover, Proposition~\ref{CxC} shows that $2$ is often in the density set of a self-product of a genus $2$ curve. In this section we construct a product $C\times D$ of two nice genus $2$ curves with a non-empty but not dense set of quadratic points, illustrating that the behavior changes in genus $2$. 

\begin{prop} \label{prop: 2notindelta}
    There exist genus $2$ curves $C$ and $D$ over $\bQ$ such that
    \begin{enumerate}
        \item $\ind(C/\mathbb{Q}) = \ind(D/\mathbb{Q})=2$,
        \item there exists a degree $2$ point on $C\times D$, and
        \item $2\notin \dendegs(C\times D/\mathbb{Q})$.
    \end{enumerate}
\end{prop}
\begin{proof}
   We claim that it is enough to find genus $2$ curves $C$ and $D$ defined over $\mathbb{Q}$ with affine models
    \[
        C: y_1^2=f(x_1),\quad \quad D: y_2^2=g(x_2)
    \]
    such that the following hold:
    \begin{enumerate}
        \item[(i)] $f(x_1) \equiv -1 \bmod 3$ for any $x_1 \in \mathbb{Z}$, and the leading coefficient is also $-1 \bmod 3$;
        \item[(ii)] $g(x_2) \equiv 1 \bmod 3$ for any $x_2 \in \mathbb{Z}$, and the leading coefficient is also $1 \bmod 3$;
        \item[(iii)] $C(\mathbb{R}) = D(\mathbb{R}) = \emptyset$;
        \item[(iv)] There exists a quadratic field $K$ for which both $C(K) \neq \emptyset$ and $D(K) \neq \emptyset$;
        \item[(v)] Both Jacobians, $\Pic^0_C$ and $\Pic^0_D$, have rank 0 over $\mathbb{Q}$.
    \end{enumerate} 
    
Suppose that $C, D$ satisfy the above assumptions. 
That $C$ and $D$ have index $2$ is an immediate consequence of the assumption $C(\bR)=D(\bR)=\emptyset$. Indeed, if $[L:\mathbb{Q}]$ is odd, then there exists at least one real place $v$ of $L$. Thus if $C(L) \ne \emptyset$, then $C(L_v)=C(\mathbb{R})$ is also nonempty, a contradiction.

Moreover since both $C$ and $D$ have a $K$-rational point, $C \times D$ does, too.
Since both curves have index $2$, we have $\dendegs(C/k)=\dendegs(D/k)=2\mathbb{N}$, and in particular both have dense quadratic points.
Quadratic points on the affine model of $C \times D$ have the form $( (x_1, \sqrt{f(x_1)}), (x_2, \sqrt{g(x_2)}))$. They either have
\begin{enumerate}
    \item both $x_1,x_2\in \mathbb{Q}$ and $k(\sqrt{f(x_1)})=k(\sqrt{f(x_2)})$ is a degree $2$ extension of $\mathbb{Q}$, or
    \item at least one of $x_1$, $x_2 \notin \mathbb{Q}$, say $x_1\not\in\bQ$, in which case $k(x_1)=k(\sqrt{f(x_1)})=k(\sqrt{f(x_2)})$ is a degree $2$ extension of $\mathbb{Q}$;
\end{enumerate}
If $x_i\in \mathbb{Q}$, then the quadratic point $(x_i,\sqrt{f(x_i)})$ is a $\mathbb{P}^1$-parametrized point (along the $x_i$-coordinate function).
Recall that since $C$ has genus $2$, a degree 2 map $C\to \mathbb{P}^1_k$ is unique up to automorphism of $\mathbb{P}^1_k$. In particular, if $x_1\notin \mathbb{Q}$, the 
quadratic point $(x_1,\sqrt{f(x_1)})$ is not $\mathbb{P}^1$-parameterized.
By the results in \S~\ref{densitybackground}, quadratic points with $x_i\notin \mathbb{Q}$ correspond to AV-parameterized points, i.e., non-zero $\mathbb{Q}$-points on the Jacobian. Since both $\Pic^0_C$ and $\Pic^0_D$ have rank $0$, there can be only finitely many such points.
In particular, points of the second type do not form a dense set of quadratic points on $C\times D$. Therefore, $2\in \dendegs(C\times D/k)$ if and only if quadratic points of the first type are dense.

Any point of the first type descends to a $\mathbb{Q}$-rational point on the surface $S$ with affine equation $z^2 = f(x_1)g(x_2)$. 
Conditions (i) and (ii) imply that if $x_1,x_2 \in \mathbb{Q}_3$, then $g(x_2)$ is a square and $f(x_1)$ is not, 
so $f(x_1)g(x_2)$ is also not a square. 
That is, $S(\mathbb{Q}_3) = \emptyset$ and thus $S(\mathbb{Q}) = \emptyset$.
Therefore, $2 \not \in \dendegs(C \times D / \mathbb{Q})$.

Finally, one can verify using Magma that the following curves satisfy conditions (i) - (v) with $K = \mathbb{Q}(i)$.
$$
C \colon y_1^2=-(x_1^2 + 1)(x_1^4 + 1), \qquad  D \colon y_2^2= -2(x_2^2+1)(x_2^4+1).
$$
\end{proof}

\begin{remark}\label{relationtoBombieri-Lang}
Let $C,D$ be two hyperelliptic curves with Weierstrass equations $y_1^2=f(x_1)$ and $y_2^2=g(x_2)$, respectively. If both curves have Jacobians with rank $0$, then we expect $2\not\in\dendegs(C\times D/k)$. Indeed, similar to the proof of Lemma~\ref{prop: 2notindelta}, the density of quadratic points on $C\times D$ is controlled by those points that arise as pullbacks from rational points on the surface $S$ with affine equation $z^2 = f(x_1)g(x_2)$. 
When $g_C, g_D\geq 2$, $S$ is a surface of general type and the Bombieri-Lang conjecture predicts that $S(k)$ is not Zariski dense. 
\end{remark}

We have a similar statement for cubic points.

\begin{prop}
    There exist genus 2 curves, $C$ and $D$, over $\mathbb{Q}$ such that \begin{enumerate}
        \item $3 \in \dendegs(C / \mathbb{Q}), \dendegs(D / \mathbb{Q})$,
        \item there exist infinitely many degree 3 points on $C \times D$, and
        \item $3 \notin \dendegs(C \times D / \mathbb{Q})$
    \end{enumerate}
\end{prop}
\begin{proof}
    We claim it is enough to construct curves $C$ and $D$ such that the following hold:
    \begin{enumerate}[label = (\roman*)]
        \item There exists unique degree $3$ morphisms $C\xrightarrow{\phi}\mathbb{P}^1$ and $D\xrightarrow{\psi}\mathbb{P}^1$.
        \item $D(\mathbb{Q}) \neq \emptyset$.
        \item For all $t \in \mathbb{Q}$, the fiber $\phi^{-1}(t)$ contains 3 points defined over $\mathbb{R}$ counted with multiplicity, and the fiber $\psi^{-1}(t)$ contains 1 real point not counted with multiplicity.
    \end{enumerate}

    The first condition implies $3 \in \dendegs(C / \mathbb{Q}), \dendegs(D / \mathbb{Q})$. The second implies the existence of infinitely many degree 3 points on $C \times D$, by considering the fiber over the rational points of $D$. It remains to prove degree 3 points are not dense on $C \times D$. Let $x=(P,Q)$ be such a point. 
    As there are at most finitely many rational points on $C$ and $D$, if degree 3 points were dense, then we may choose $x$ so that both $P$ and $Q$ are cubic points; in particular the residue fields $k(P), k(Q)$ are the same.  
    Since $\phi$ is the unique source of degree 3 points on $C$, and all cubic points coming from $\phi$ are totally real, $k(P)$ must also be totally real. However, since $\psi$ is the unique source of degree 3 points on $D$, and all cubic points coming from $\psi$ are not totally real, we have a contradiction. 

    It remains to prove such curves exist. Let $C$ be $y^2 = -(x^3 - 6x -1)(x^3 - 6x + 2)$. A calculation in Magma shows that $C$ has trivial Mordell-Weil group, thus all degree 3 points are linearly equivalent to the sum of the points with $x^3 - 6x - 1 =0$. The corresponding map to $\mathbb{P}^1$ is $(x, y) \mapsto \frac{y}{x^3 - 6x - 1}$. The fiber above a rational point, $t$, is given by the points satisfying $y = t(x^3 - 6x - 1)$ and $t^2 (x^3 - 6x - 1) = -(x^3 - 6x + 2)$. These will be totally real since the discriminant of this cubic is $837t^8 + 3510t^6 + 5265t^4 + 3348t^2 + 756 > 0$.
    Let $D$ be $y^2 = x^6 + 2x^4 + 2x^3 + x^2 + 2x + 2$. This has two rational points at infinity, their difference is 3-torsion and generates the Mordell-Weil group. This shows any degree 3 divisor is linearly equivalent to a sum of the points at infinity. Only one of these divisors is base-point free, and it has section $y - (x^3 + x + 1)$. The fiber above a rational point, $t$, is given by the points satisfying $y = (x^3 + x + 1) + t$, and $2t(x^3 + x + 1) + t^2 = 1$. The discriminant of this is 
    $$-108t^6 - 432t^5 - 280t^4 + 432t^3 - 108t^2 = -64t^4 - 108(t^3 + 2t^2 - t)^2.$$
    This is negative, except at $t = 0$ and $\infty$, which correspond to the two rational points.
\end{proof}

\subsection{Some results about potential density degrees}
Recall that for a nice variety $X$ over a number field $k$, the potential density degree set $\wp(X / k)$ is the union of $\delta\left(X' / k^{\prime}\right)$ as $k^{\prime} / k$ ranges over all possible finite extensions (see \Cref{def:potentialdensity}). 

The results obtained earlier in this section allow us to obtain very concrete information on $\wp(C\times D/k)$ in the case of a product of genus $2$ curves.
First of all, let us see how \Cref{bounds} reads for the potential density case.
\begin{prop}\label{prop:potdegdens}
Let $C,D$ be nice curves of genus $g_C,g_D$ over $k$.
Then $\wp(C\times D / k) \subseteq \potdendegs(C / k) \cap \wp(D / k)$.
Moreover if $g_C,g_D\leq 9$ then 
\[
\potdendegs(C / k) \cdot \wp(D / k) \subseteq \wp(C\times D / k) \subseteq \potdendegs(C / k) \cap \wp(D / k).
\]
\end{prop} 
\begin{proof}

By \Cref{bound1}, we get containments $\potdendegs(C \times D / k) \subset \potdendegs(C / k)$ and $\potdendegs(C \times D / k) \subset \potdendegs(D / k)$. Combining these two gives the upper bound.

When the genus of $C$ or $D$ is less than 9, we establish the lower bound by reducing to the case for density degree sets. By \cite[Proposition 5.5.1]{VV}, there exist Galois fields $k_C$, $k_D$ such that $\potdendegs(C / k) = \dendegs(C / k_C)$ and similarly for $D$. Moreover, by the proof of the same proposition we may assume that $k_C, k_D$ are maximal possible in the sense that if $k'/k_C$ is a Galois extension, then $\dendegs(C/k_C)=\dendegs(C/k')$ and the same for $D$. As $k_C$ and $k_D$ are Galois, their compositum, $k'$, is also Galois. By the definition of $\potdendegs(C / k)$ and the stability of $\dendegs$ under Galois field extension, $\dendegs(C / k') = \potdendegs(C / k)$ and $\dendegs(D / k') = \potdendegs(D / k)$. Over $k'$, \Cref{inclusionsprod} implies that $\dendegs(C / k') \cdot \dendegs(D / k') \subseteq \dendegs(C \times D / k') \subseteq \potdendegs(C \times D / k)$.  
\end{proof}

We now specialize to the case when $C, D$ are nice genus 2 curves over $k$. 

\begin{coroll}\label{cor:lowerboundonpotdendegscurvesgenus2}
    Let $C,D$ be nice genus 2 curves over $k$. Then
    \[
    \mathbb{N}_{\geq 2} \setminus \{2,3,5,7,11\} \subseteq \potdendegs(C\times D/k).
    \]
\end{coroll} 
\begin{proof} 
First we observe that for a nice genus $2$ curve $C$ it follows by the second part of \Cref{lemma:deltagenus2} that $\{2\}\sqcup \mathbb{N}_{\geq 4} \subseteq \potdendegs(C/k)$ (enough to base change to a $k'$ where $\ind(C'/k')=1$).
Moreover, by \cite[Lemma 4.1.4]{VV}, $3\in \delta(C'/k')$ if  there is a degree 3 divisor $D$ on $C'$ such that $h^0(D)\geq 2$ and $|D|$ is base-point free.
Let $k'$ be such that there exists a non-Weierstrass $k'$-point $P$.
Set $D=3P$: then $h^0(3P)= 2$ and $h^0(3P-P') = h^{0}(2P)-1$ for all $P'\in C(\overline{k})$.
Hence is base-point free as well.
Therefore $\potdendegs(C/k)=\mathbb{N}_{\geq 2}$.
    
Thus, we can find a finite extension $k'/k$ such that $C'(k'), D'(k') \neq \varnothing$ and such that $\dendegs(C'/k')=\dendegs(D'/k')=\mathbb{N}_{\geq 2}$.
Then by \Cref{prop:summary1}, we have $\mathbb{N}_{\geq 2} \setminus \{$primes $\leq 11\}\subseteq \dendegs(C'\times D'/k')$.
Hence the statement.
\end{proof}

The next Theorem shows that the lower bound in the above Corollary can be significantly improved.

\begin{thm} Let $C,D$ be nice genus $2$ curves over $k$. Then,
\[\mathbb{N}_{\geq 2} \setminus \{2, 3, 5\} \subseteq \potdendegs(C\times D/k).\]
\label{thm:pot}
\end{thm}
\begin{proof}
After base changing to a finite extension $k'$, it can be assumed that the affine models for the curves are
\[
C: y_1^2 = f(x_1), \quad D: y_2^2 = g(x_2),
\]
where $f$ and $g$ are both monic degree 5 polynomials, and moreover, both have a rational root, which may be assumed to be at $x = 0$.
Consider the family of curves, $X_{\gamma}$, as in Corollary~\ref{coroll:inclusionofN(C,D)}, which have arithmetic genus 9. These curves have at least 2 nodes: one lying over the points of $C$ and $D$ at infinity, and one over the points of $C$ and $D$ at $x = 0$. For general $\gamma$, there are no other singularities. 
This shows that the general member of $X_{\gamma}$ has geometric genus 7. 
In particular, for infinitely many $\gamma$, the normalization of $X_{\gamma}$ has at least 2 rational points (lying above $\infty$) by \Cref{lem:split_over_node}. Label these points $\infty_{\pm}$. 
Again by Castelnuovo--Severi inequality, $X_\gamma$'s are not hyperelliptic: then by \Cref{curve_asymptotic}, the existence of $\infty_{\pm}$ shows that 11 is in the density degree set as well.

It remains to handle 7.
Consider the degree 2 effective divisor $P$ lying over $x = 0$: it is the union of two distinct $\overline{k}$-points, which we label by $P_1, P_2$ (note that they may not be defined over $k$).
Let $D$ be the degree 8 divisor on the normalization $X_\gamma^\nu$ defined as $ 3\infty_+ + 3\infty_- + P$. 
It has at least 3 non-trivial sections, $x_1$, $\frac{y_1}{x_1}$ and $\frac{y_2}{x_2}$. These give a birational map from $X_\gamma$ into $\mathbb{P}^3$, which is naturally an embedding away from $\infty_{\pm}$ and the $P_i$. For general $\gamma$, it is injective on these points too (by considering the values of $\frac{y_1}{y_2}$ at these points), and has smooth image. In particular, as $D$ is a base-point free divisor and an embedding, $D - \infty_+$ is a base-point free divisor of degree 7. This gives $7 \in \potdendegs(C \times D / k)$.
\end{proof}

\begin{remark}
    These bounds can be realized over $k$ as density degree sets when the assumptions on the extension $k'$ are satisfied by $k$ in the above proof. The proof requires that $C$ and $D$ both have degree 3 points and that they have at least 2 rational Weierstrass points. For example, the curves $y^2 + (x^3 + 1)y = -x^5 + x^3 + x^2 + 3x + 2$ (LMFDB label \href{https://www.lmfdb.org/Genus2Curve/Q/249/a/6723/1}{249.a.6723.1}) and $y^2 + y = 2x^5 - 3x^4 + x^3 + x^2 - x$  (LMFDB label \href{https://www.lmfdb.org/Genus2Curve/Q/256/a/512/1}{256.a.512.1}) satisfy these assumptions over $\mathbb{Q}$, as they have at least one non-Weierstrass rational point, as well as the 2 rational Weierstrass points. 
    On the other hand, we can find examples where the lower bound for the potential density set $\wp(C\times D/k)$ can be improved even more. Namely,
     it follows by Proposition~\ref{CxC} that $2\in\wp(C\times C/k)$, when $C$ has geometrically simple Jacobian, as we can find a finite extension $k'/k$ over which $\Pic^0_{C_k'}$ has positive rank. For the same reason, if $C, D$ are the genus $2$ curves considered in Example~\ref{ex:isomorphicJacobians}, it follows that $3\in\wp(C\times D/k)$.
\end{remark}

\section{Principally Polarized Abelian Surfaces}

We conclude this paper by giving some results about the density degree set of some abelian varieties. 

\begin{prop}\label{prop:simpleabeliansurface}
    Let $A$ be an abelian variety isomorphic to the product $B_1\times \dots \times B_r$ of simple abelian varieties over $k$. 
    Then 
    \[
    1\in \dendegs(A/k)
    \ 
    \Leftrightarrow 
    \
    1\in \dendegs(B_i/k) \ \forall i=1,\dots, r \
    \Leftrightarrow 
    \textup{ rank }B_i(k)>0\ \forall i=1,\dots, r.
    \]
\end{prop}
\begin{proof}
If the $k$-points are dense in $A$, projecting on each factor we get that rational points are dense on each $B_i$.
Conversely if the $k$-points are dense in each factor, the Zariski closure of the $k$-points of $A$ must be everything.
So the first equivalence is proven and we can focus on the second one. 
It is enough to deal with the case where $A$ is simple.

The forward direction is clear. 
Let us then assume by contradiction that $A$ has positive rank but that $A(k)$ is not dense.
Let $W$ be the Zariski closure of $A(k)$ in $A$ (with reduced structure on it) which is the union of finitely many integral components $W=\bigcup_i W_i$.
By definition each $W_i(k)$ is dense in $W_i$.
Therefore by Faltings--Vojta Theorem (see \cite[Cor. of Thm. 3.1]{Maz00}), for each $W_i$ there exist finitely many $\{a_{ij}\}_j\subset A(k)$ and $B_{ij}\subseteq A$ abelian subvarieties such that 
\[
W_i = \bigcup_j ( a_{ij}+ B_{ij} ).
\]
Since $A$ is simple, then each $B_{ij}$ is either a point or $A$ itself. 
Therefore $W$ is either the union of finitely many points or $A$ itself, a contradiction.
\end{proof}

\begin{remark}
We note that a result of Hassett \cite[Part II, Proposition 4.2]{Hassett2003}  uses a similar argument to show that for an arbitrary abelian variety $A$ over a number field $k$ the potential density degree set $\potdendegs(A/k)$ always contains $1$.
\end{remark}

\begin{thm}\label{thm:isogeneousabsurf}
Let $A$ and $B$ be two isogenous abelian surfaces with $B$ principally polarized. Then $\dendegs(A/k)=\dendegs(B/k)$. 
\end{thm}
   
\begin{proof}
Let us first deal with the case when $A$ is not simple. 
Without loss of generality we may assume $B=E_1\times E_2$.
Let $\gamma:A\to E_1\times E_2$ be an isogeny (defined over $k$) and let $m\geq 1$ be its degree. 
Then there exists an isogeny $\mu:A\to E_1\times E_2$ so that $\gamma\circ\mu=[m]$. We first claim that $\dendegs(E_1\times E_2/k)\subseteq\dendegs(A/k)$. 
    
If $d\in\{3\}\cup\nn_{\geq 5}$, then Corollary~\ref{cor:k(P1,P2)=k(nP1,nP2)} yields that there exists a point $x=(P_1,P_2)\in E_1\times E_2$ with both $P_i$ of degree $d$ and of infinite order such that $k(P_i)=k([n]P_i)$ for every $n\geq 1$. Since  we have inclusions $k([m]x)\subseteq k(\mu(x))\subseteq k(x)$, it follows that $\mu(x)$ has degree $d$, and more generally the set $\{\mu([n_1]P_1, [n_2]P_2)\}$ is a dense subset of $A$ consisting of elements of degree $d$. 
    
If $d=4$, we can consider a biquadratic point $(P_1,P_2)\in E_1\times E_2$ with $P_1, P_2$ of infinite order such that $P_i$ is $\bZ$-linearly independent from $E_i(k)$ as done in the proof of Corollary~\ref{cor:simultaneousrankjumps}.
Then as before, $k([m]P_1,[m]P_2)=k(P_1,P_2)$, and the set $\{ \mu( [n_1]P_1, [n_2]P_2 )  \}$ yields a dense subset of $A$ of degree $4$ points.

Let us now consider the case when $d=2\in \dendegs(E_1\times E_2/k)$.
If both elliptic curves have null rank, then the image of the degree 2 points via the map $\mu\circ \gamma$ must contain a dense set of degree 2 points: therefore also the image of $\gamma$ does.

If at least one of the two elliptic curves has positive rank (wlog $E_1$), we can consider points $x=(P_1,P_2) \in E_1\times E_2$ of degree $(1,2)$ such that $P_2$ has infinite order and it is $\bZ$-linearly independent from $E_2(k)$.
We then conclude as before.

If $d=1$, then the $k$-points are dense in $E_1\times E_2$ and so are in $A$. Hence we have proven the inclusion.

Let us look at the other inclusion  $\dendegs(A/k)\subseteq\dendegs(E_1\times E_2/k)$. 
Since $\dendegs(E_1\times E_2/k)$ contains all the integers greater or equal than 3, the only thing left to check is that $1,2\in\dendegs(E_1\times E_2/k)$ when $1,2\in\dendegs(A/k)$. 
For $d=1$ the claim is clear.

If $2\in \dendegs(A/k)$ let $Z_2(A) = \{ x\in A : \deg(x)=2 \}$.
Then $\gamma_*(Z_2(A))$ is the union $H_1\sqcup H_2$ where
    \[
    H_d = \{ \gamma(x) : x\in Z_2(A) \textup{ and } [k(\gamma(x)):k]=d  \}.
    \]
    Therefore, either $H_1$ or $H_2$ is dense in $E_1\times E_2$. 
    Since $\dendegs(E_1\times E_2/k)$ is multiplicative closed, $2\in \dendegs(E_1\times E_2/k)$. 
 
    Let now assume that $A$ is simple and hence $B$ is the Jacobian $\Pic^0_C$ of a nice genus 2 curve over $k$.
    Again, let $\gamma:A\to \Pic^0_C$ be an isogeny. 
    If the $k$-points are dense in $A$, they are in $\Pic^0_C$ as well.
    By Theorem~\ref{thm:Jac of genus 2} $\dendegs(\Pic^0_C/k)$ contains $\mathbb{N}_{\geq 2}$: thus $\dendegs(A/k)\subseteq \dendegs(\Pic^0_C/k)$.
    
    Let $\mu:\Pic^0_C\to A$ be an isogeny so that $\gamma\circ \mu = [m]_{\Pic^0_C}$, where $m=\deg(\mu)$.
    If $\Pic^0_C$ has dense $k$-points, so does $A$.
    Recall now that if $\Pic^0_C$ contains a curve $D$ of (geometric) genus at least 2, then $\{[n](D)\}_{n\in \mathbb{Z}}$ is dense in $\Pic^0_C$.
    More precisely $\{[nm](D)= [n]([m]D)\}_{n\in \mathbb{Z}}$ is dense in $\Pic^0_C$.
    Moreover by proof of Theorem~\ref{thm:Jac of genus 2} an exhaustive source of density for $\Pic^0_C$ in degree at least 2 comes from the images under the multiplication maps $\{[nm]\}_{n\in \mathbb{Z}}$ of $C\xhookrightarrow{} \Pic^0_C, x\mapsto \mathcal{O}_C(2x-\omega_C)$ and an arithmetic genus 4 curve $D\subset \Pic^0_C$ with dense degree 3 points and index 1 normalization $D^{\nu}$.

    In particular (either using $C$ or $D$) for any degree $d\geq 2$, we have a dense set of degree $d$ points on $\Pic^0_C$ such that the images along $[m]_{\Pic^0_C}$ are still of degree $d$. 
    We then conclude as in the product $E_1\times E_2$ case.
\end{proof}

\begin{remark}
    These results are not exhaustive for general abelian surfaces. Whilst every abelian surface is isogenous to a principally polarised abelian surface over an algebraically closed field, this is not necessarily true over the original field. 
    For example, let $E$ be the elliptic curve $y^2 + y = x^3 - x^2$ over $\mathbb{Q}$. The Weil restriction of $E$ from $\mathbb{Q}(\zeta_7)^+$ to $\mathbb{Q}$ is equipped with a trace map to $E$, and the kernel is an abelian surface, $A$. 
    Theorem 3 in \cite{hmnr08} shows that the reduction of $A$ modulo $p$ will not be isogenous to a principally polarized abelian variety if:
    \begin{enumerate}
        \item $p - a_p^2 > 0$,
        \item all prime factors of $p - a_p^2$ are 1 modulo 3, and
        \item $p$ is totally inert in $\mathbb{Q}(\zeta_7)^+$,
    \end{enumerate}    
    where $a_p$ denotes the trace of Frobenius of $E$ at $p$. 
    The prime $17$ satisfies these assumptions since $a_{17} = -2$. If $A$ were isogenous to a principally polarized abelian variety $B$ over $\mathbb{Q}$, then their reductions at 17 would be isogenous, and the reduction of $B$ would remain principally polarized, contradicting that the reduction of $A$ at 17 is not isogenous to a principally polarized abelian variety.
\end{remark}

\printbibliography
\end{document}